\documentclass{amsart}[12pt]
\usepackage{graphicx}
\usepackage{amssymb}
\usepackage{amsmath,amsthm}
\usepackage{hyperref}
\usepackage{amsthm}
\usepackage[english]{babel}
\usepackage{mathrsfs}

\newtheorem{main}{Theorem}

\newtheorem*{OzSolid}{Ozawa's Solidity Theorem}
\newtheorem*{PetSolid}{Peterson's Solidity Theorem}
\newtheorem{Corollaryb1}{Corollary B.\!\!}
\newtheorem*{CorollaryC}{Corollary C}
\newtheorem{theorem}{Theorem}[section]
\newtheorem{lem}[theorem]{Lemma}
\newtheorem{prop}[theorem]{Proposition}

\theoremstyle{definition}
\newtheorem{definition}[theorem]{Definition}

\theoremstyle{remark}
\newtheorem{remark}[theorem]{Remark}
\newtheorem{question}[theorem]{Question}

\def\ca{\curvearrowright}
\def\ra{\rightarrow}
\def\e{\epsilon}
\def\G{\Gamma}
\def\g{\gamma}
\def\la{\lambda}
\def\La{\Lambda}

\def\a{\alpha}
\def\s{\sigma}
\def\bb{\mathbb}
\def\fr{\mathfrak}
\def\oo{\overline{\otimes}}
\def\up{\upsilon}

\def\vp{\varphi}
\def\de{\delta}
\def\del{\partial}
\def\Sg{\Sigma}
\def\Cal{\mathcal}

\DeclareMathOperator*{\id}{{id}}
\DeclareMathOperator*{\Lim}{{Lim}}

\numberwithin{equation}{section}

\newcommand{\ip}[2]{\langle #1, #2 \rangle}

\newcommand{\abs}[1]{\lvert#1\rvert}
\providecommand{\nor}[1]{\lVert #1 \rVert}

\begin{document}

\title[${\rm II}_1$ factors of negatively curved groups]{On the structural theory of ${\rm II}_1$ factors of negatively curved groups\\ Sur la structure des facteurs de type $\rm II_1$ associ\'e avec les groupes de courbure n\'egative}

\author{Ionut Chifan}
\address{Department of Mathematics, 1326 Stevenson Center,  Vanderbilt University, Nashville, TN 37240, USA and IMAR, Bucharest, Romania} 
\curraddr{Department of Mathematics, University of Iowa, 14 MacLean Hall, Iowa City, IA 52242-1419, USA}
\email{ionut-chifan@uiowa.edu}
\thanks{The first author was supported in part by NSF Grant \#101286.}

\author{Thomas Sinclair}
\address{Department of Mathematics, 1326 Stevenson Center, Vanderbilt University, Nashville, TN 37240, USA}
\curraddr{Department of Mathematics, University of California, Los Angeles, Box 951555, Los Angeles, CA, 90095-1555, USA}
\email{thomas.sinclair@math.ucla.edu}

\subjclass[2010]{46L10; 20F67}

\date{\today}

\dedicatory{}

\keywords{strong solidity, negatively curved groups, bi-exact groups}

\begin{abstract} Ozawa showed in \cite{OzSolid} that for any i.c.c.\ hyperbolic group, the associated group factor $L\G$ is solid.  Developing a new approach that combines some
methods of Peterson \cite{PetL2},  Ozawa and Popa \cite{OPCartanI,OPCartanII}, and Ozawa \cite{OzCBAP}, we strengthen this result by showing that
$L\G$ is strongly solid. Using our methods in cooperation with a cocycle superrigidity result of Ioana \cite{IoaCSR}, we show that profinite actions of lattices in ${\rm Sp}(n,1)$, $n\geq 2$, are virtually $W^*$-superrigid. 
\vskip 0.05in
\noindent \emph{Keywords:} strong solidity, negatively curved groups, bi-exact groups\\

\noindent \textsc{R\'esum\'e}. Ozawa \`a montr\'e dans \cite{OzSolid} soit une c.c.i.\ groupe hyperbolique, le facteur de type $\rm II_1$ associ\'e est solide. En developpant une nouvelle approche qui en combine les m\'ethodes de Peterson \cite{PetL2}, Ozawa et Popa \cite{OPCartanI, OPCartanII}, et Ozawa \cite{OzCBAP}, nous renforcent ce r\'esultat en montrant cela facteur est fortement solide. Suivre nous m\'ethodes en cooperation avec une r\'esultat d'Ioana de superrigidit\'e des cocycles \cite{IoaCSR}, nous prouvent que les actions des r\'eseaux de ${\rm Sp}(n,1)$, $n\geq 2$,  sont virtuellement $\rm W^*$-superrigide.
\vskip 0.05in
\noindent \emph{Mots-clefs:} fortement solidit\'e, groupes de courbure n\'egative, groupes \guillemotleft bi-exact\guillemotright

\end{abstract}

\maketitle

\section*{Introduction}

In a conceptual leap Ozawa established a broad property for group factors of Gromov hyperbolic groups---what he termed
solidity---which essentially allowed him to reflect the ``small cancellation'' property such a group enjoys in terms of its associated von Neumann algebra.

\begin{OzSolid}[\cite{OzSolid}] If $\G$ is an i.c.c.\ Gromov hyperbolic group, then $L\G$ is \textit{solid},
i.e., $A' \cap L\G$ is amenable for every diffuse von Neumann subalgebra $A \subset L\G$.
\end{OzSolid}

\noindent Notable for its generality, Ozawa's argument relies on a surprising interplay between ${\rm C}^\ast$-algebraic and von Neumann algebraic techniques \cite{BrOz}.

Using his deformation/rigidity theory \cite{PoICM}, Popa was able to offer an alternate, elementary proof of solidity for free group factors: more generally, for factors admitting a ``free malleable deformation'' \cite{PoFree}. Popa's approach exemplifies the use of spectral gap rigidity arguments that opened up many new directions in deformation/rigidity theory, cf.\ \cite{PoICM,PoFree,PoSG}. Of particular importance, these techniques brought the necessary perspective for a remarkable new approach to the Cartan problem for free group factors in the work of Ozawa and Popa \cite{OPCartanI,OPCartanII}---an approach which this work directly builds upon.

A new von Neumann-algebraic approach to solidity was developed by Peterson in his important paper on $L^2$-rigidity \cite{PetL2}.
Essentially, Peterson was able to exploit the ``negative curvature'' of the free group on two generators $\bb F_2$, in terms of a proper
1-cocycle into the left-regular representation, to rule out the existence of large relative commutants of diffuse subalgebras of $L\bb F_2$.

\begin{PetSolid}[\cite{PetL2}] If $\G$ is an i.c.c.\ countable discrete which admits a proper 1-cocycle $b:
\G\to \Cal H_\pi$ for some unitary representation $\pi$ which is weakly-$\ell^2$ (i.e., weakly contained in the left-regular representation),
then $L\G$ is solid.
\end{PetSolid}

\noindent It was later realized by the second author \cite{Sin} that many of the explicit
unbounded derivations (i.e., the ones constructed from 1-cocycles) that Peterson works with have natural dilations which are malleable
deformations of their corresponding (group) von Neumann algebras.

However, the non-vanishing of 1-cohomology of $\G$ with coefficients in the left-regular representation does not reflect the full spectrum of
negative curvature phenomena in geometric group theory as evidenced by the existence of non-elementary hyperbolic groups with Kazhdan's
property (T), cf. \cite{BV}. In their fundamental works on the rigidity of group actions \cite{MScocycle, MSoe}, Monod and Shalom proposed a
more inclusive cohomological definition of negative curvature in group theory which is given in terms of non-vanishing of the second-degree
bounded cohomology for $\G$ with coefficients in the left-regular representation. Relying on Monod's work in bounded cohomology \cite{Monod},
we will make use of a related condition, which is the existence of a proper \textit{quasi}-1-cocycle on $\G$ into the left-regular
representation (more generally, into a representation weakly contained in the left-regular representation), cf.\ \cite{Monod, Tho}. By a result
of Mineyev, Monod, and Shalom \cite{MMS}, this condition is satisfied for any hyperbolic group---the case of vanishing first $\ell^2$-Betti
number is due to Mineyev \cite{Min}.

\subsection*{Statement of results} We now state the main results of the paper, in order to place them within the context of previous results in the structural theory of group von Neumann algebras. We begin with the motivating result of the paper, which unifies the solidity theorems of Ozawa and Peterson.

\begin{main}\label{solidity} Let $\G$ be an i.c.c.\ countable discrete group which is exact and admits a proper quasi-1-cocycle $q: \G\to \Cal H_\pi$ for some weakly-$\ell^2$ unitary representation $\pi$ (more generally, $\G$ is exact and belongs to the class $\Cal{QH}_{\rm reg}$ of Definition \ref{defn:QH}). Then $L\G$ is solid.
\end{main}

\noindent In particular, all Gromov hyperbolic groups are exact, cf. \cite{Roe}, and admit a proper quasi-1-cocycle for the left-regular
representation \cite{MMS}. For the class of exact groups, belonging to the class $\Cal{QH}_{\rm reg}$ is equivalent to bi-exactness (see section \ref{sec:cohom}), so the above result is equivalent to Ozawa's Solidity Theorem.

Following Ozawa's and Peterson's work on solidity, there was some hope that similar techniques could be used to approach to the Cartan subalgebra problem for group factors of hyperbolic groups,
generalizing Voiculescu's celebrated theorem on the absence of Cartan subalgebras for free group factors \cite{Voic}.  However, the Cartan
problem for general hyperbolic groups would remain intractable until the breakthrough approach of Ozawa and Popa through Popa's deformation/rigidity
theory resolved it in the positive for the group factor of any discrete group of isometries of the hyperbolic plane \cite{OPCartanII}. In fact,
they were able to show that any such ${\rm II}_1$ factor $M$ is \textit{strongly solid}, i.e., for every diffuse, amenable von Neumann
subalgebra $A\subset M$, $ \Cal N_M(A)''\subset M$ is an amenable von Neumann algebra, where $\Cal N_M(A) = \{u\in \Cal U(M) : uAu^* = A\}$.

Using the techniques developed by Ozawa and Popa \cite{OPCartanI, OPCartanII} and a recent result of Ozawa \cite{OzCBAP}, we obtain
the following strengthening of Theorem \ref{solidity}.

\begin{main}\label{strongsolidity} Let $\G$ be an i.c.c.\ countable discrete group which is weakly amenable (therefore, exact). If $\G$ admits a proper quasi-1-cocycle into a weakly-$\ell^2$ representation, then $L\G$ is strongly solid.
\end{main}

\noindent Appealing to Ozawa's proof of the weak amenability of hyperbolic groups \cite{OzHyp}, Theorem \ref{strongsolidity} allows us to fully
resolve in the positive the strong solidity problem---hence the Cartan problem---for i.c.c.\ hyperbolic groups and for lattices in connected
rank one simple Lie groups. In particular, if $\G$ is an i.c.c.\ lattice in ${\rm Sp}(n,1)$ or the exceptional group ${\rm F}_{4(-20)}$, then $L\G$ is strongly solid. The strong solidity problem for the other rank one simple Lie groups---those locally isomorphic to ${\rm SO}(n,1)$ or ${\rm SU}(n,1)$)---was
resolved for ${\rm SO}(2,1)$, ${\rm SO}(3,1)$, and ${\rm SU}(1,1)$ by the work of Ozawa and Popa \cite{OPCartanII} and, in the
general case, by the work of the second author \cite{Sin}. The results follow directly from Theorem \ref{strongsolidity} in the co-compact
(i.e., uniform) case: in the non-uniform case, we must appeal to a result of Shalom (Theorem 3.7 in \cite{Shalom}) on the integrability of
lattices in connected simple rank one Lie groups to produce a proper quasi-$1$-cocycle.

Building on Ioana's work on cocycle superrigidity \cite{IoaCSR}, we are also able to obtain new examples of
virtually ${\rm W}^\ast$-superrigid actions.

\begin{Corollaryb1}\label{Sp} Let $\G$ be an i.c.c.\ countable discrete group which is weakly amenable and which admits a proper quasi-$1$-cocycle into a weakly-$\ell^2$ representation. If $\G\ca (X,\mu)$ is a profinite, free, ergodic measure-preserving action of $\G$ on a standard probability space $(X,\mu)$, then $L^\infty(X,\mu)\rtimes \G$ has a unique Cartan subalgebra up to unitary conjugacy. If in addition $\G$ has Kazhdan's property (T) (e.g., $\G$ is a lattice in ${\rm Sp}(n,1)$, $n\geq 2$), then any such action $\G\ca (X,\mu)$ is virtually ${\rm W}^\ast$-superrigid.
\end{Corollaryb1}

A natural question to ask is whether our techniques can be extended to demonstrate strong
solidity of the group factor of any i.c.c.\ countable discrete group which is relatively hyperbolic \cite{Osin} to a family of amenable
subgroups.

The techniques used to prove Theorem \ref{strongsolidity} also allow us to deduce, by way of results of Cowling and Zimmer \cite{CoZi} and Ozawa \cite{OzHyp}, the following improvement of results of Adams (Corollary 6.2 in \cite{AdHyp}) and of Monod and Shalom (Corollary 1.19 in \cite{MSoe}) on the structure of groups which are orbit equivalent to hyperbolic groups.

\begin{Corollaryb1} \label{generalizadams} Let $\G$ be an i.c.c.\ countable discrete group which is weakly amenable and which admits a proper quasi-$1$-cocycle into a weakly-$\ell^2$ representation. Let $\G\ca (X,\mu)$ be a free, ergodic, measure-preserving action of $\G$ on a probability space and $\La\ca (Y,\nu)$ be an arbitrary free, ergodic, measure-preserving action of some countable discrete group $\La$ on a probability space. If $\G\ca (X,\mu)$ is orbit equivalent to $\La\ca (Y,\nu)$, then $\La$ is not isomorphic to a non-trivial direct product of infinite groups and the normalizer of any infinite, amenable subgroup $\Sigma < \La$ is amenable.
\end{Corollaryb1}

Beyond solidity results, we highlight that the techniques developed in this paper also enable us to reprove strong decomposition results for products of groups in the spirit of Popa's deformation/rigidity theory. Specifically, we are able to recover the following prime decomposition theorem of Ozawa and Popa.

\begin{main}[Ozawa and Popa \cite{OPPrime}]\label{prime} Let $\G = \G_1 \times \dotsb \times \G_n$ be a non-trivial product of exact, i.c.c.\ countable discrete groups such that $\G_i\in \Cal{QH}_{\rm reg}$, $1 \leq i \leq n$. If $N = N_1 \oo \dotsb \oo N_m$ is a product of $\rm II_1$ factors $N_j$, $1 \leq j \leq m$, for some $m\geq n$, and $L\G \cong N$, then $m = n$ and there exist $t_1,\dotsc,t_n >0$ with $t_1\dotsb t_n = 1$ so that, up to a permutation, $(L\G_i)^{t_i} \cong N_i$, $1 \leq i \leq n$.
\end{main}

An advantage to our approach is that our proof naturally generalizes to unique measure-equivalence decomposition of products of bi-exact groups, first proven by Sako (Theorem 4 in \cite{Sako}). This type of result was first achieved for products of groups of the class $\Cal C_{\rm reg}$ by Monod and Shalom (Theorem 1.16 in \cite{MSoe}).

\begin{CorollaryC}[Sako \cite{Sako}]\label{meprime} Let $\G = \G_1 \times \dotsb \times \G_n$ be a non-trivial product of exact, i.c.c.\ countable discrete groups such that $\G_i\in\Cal {QH}_{\rm reg}$, $1 \leq i \leq n$, and let $\La = \La_1\times \dotsb \times \La_m$ be a product of arbitrary countably infinite discrete groups. Assume that $\G \sim_{\rm ME} \La$, i.e., there exist $\G\ca (X,\mu)$ and $\La\ca (Y,\nu)$ free, ergodic, probability measure-preserving actions which are weakly orbit equivalent (Definition 2.2 in \cite{FurOE}). If $m\geq n$ then $m = n$ and, up to permuting indices, we have that $\G_i \sim_{\rm ME} \La_i$, $1 \leq i \leq n$.
\end{CorollaryC}

\subsection*{On the method of proof} This paper began as an attempt to chart a ``middle path'' between the solidity theorems of Ozawa, Popa, and Peterson by recasting Ozawa's approach to solidity effectively as a deformation/rigidity argument. We did so by finding a ``cohomological'' characterization of Ozawa's notion of bi-exactness \cite{OzKurosh}. Interestingly, our reformulation of bi-exactness has many affinities with (strict) cohomological definitions of negative curvature proposed by Monod and Shalom \cite{MSoe} and Thom \cite{Tho}.

Working from the cohomological perspective, we were able construct ``deformations'' of $L\G$. Though these
``deformations'' no longer mapped $L\G$ into itself, we were still able to control their convergence on a weakly dense ${\rm
C}^\ast$-subalgebra of $L\G$ namely, the reduced group ${\rm C}^\ast$-algebra ${\rm C}_\la^\ast(\G)$, then borrow Ozawa's insight
of using local reflexivity to pass from ${\rm C}_\la^\ast(\G)$ to the entire von Neumann algebra.

After this initial undertaking had been completed, we turned our attention to applying these techniques to the foundational methods which Ozawa and Popa developed in proving strong solidity of free group factors. Our approach through deformation/rigidity-type arguments allowed us to exploit the ``compactness'' of deformations coming from quasi-cocycles to achieve a finer degree of control than is afforded by the use of bi-exactness. This extra control was crucial in our adaptation of Ozawa and Popa's  fundamental techniques in the proof of Theorem \ref{strongsolidity}. This should be considered the main technical advance of the paper.

\subsection*{Acknowledgements} We thank Jesse Peterson for valuable discussions on this work. His comments at an early stage of this project were instrumental in helping this work to assume its final form. The second author extends his gratitude to Rufus Willett for several interesting conversations around Ozawa's solidity theorem. We also warmly thank Dietmar Bisch, Adrian Ioana, Narutaka Ozawa, Sorin Popa, and Stefaan Vaes for their useful comments and suggestions regarding this manuscript. We finally thank the anonymous referee, whose numerous comments and suggestions greatly improved the exposition.

\section{Cohomological-type properties and negative curvature}\label{sec:cohom} Let $\G$ be a countable discrete group. Recall that a \textit{length function} $\abs{\,\cdot\,}: \G\to \bb R_{\geq 0}$ is a map satisfying: (1) $\abs{\g} = 0$ if and only if $\g = e$ is the identity; (2) $ \abs{\g^{-1}} = \abs{\g}$, for all $\g\in \G$; and (3) $\abs{\g\de} \leq \abs{\g} + \abs{\de}$, for all $\g,\de\in \G$.

\subsection{Arrays} We introduce a general class of embeddings of a group $\G$ into Hilbert space that are compatible with some action of $\G$ by orthogonal transformations, which we refer to as arrays. These ``arrays'' distill the essential structural properties of proper affine isometric actions while adding a substantial amount of ``geometric'' flexibility. In fact, the simplest example of an array will be a length function, which can be thought of as taking values in the trivial orthogonal representation.

\begin{definition}\label{def:array} Let $\pi: \G\to  \Cal O( \Cal H_\pi)$ be an orthogonal representation of a countable discrete group $\G$ and let $\mathcal G$ be a family of subgroups of $\G$. A map $q: \G\to \Cal  H_\pi$ is called an \textit{array}  for every finite subset $F\subset \G$ there exists $K\geq 0$ such that
\begin{equation}\label{weakquasi} \nor{\pi_\g(q(\de)) - q(\g\de)}\leq K,\end{equation} for all $\g\in F$, $\de\in \G$ (i.e., $q$ is \textit{boundedly equivariant}). It is an easy exercise to show that for any array $q$ there exists a length function on $\G$ which bounds $\nor{q(\g)}$ from above. An array $q: \G\to \Cal  H_\pi$ is said to be:
\begin{itemize}
\item \textit{proper with respect to $\mathcal G$} if the map $\g\mapsto \nor{q(\g)}$ is proper with respect to the family $\mathcal G$, i.e., for all $C>0$ there exist finite subsets $G,H\subset \G$, $\mathcal K\subset \mathcal G$ such that  $$\{\g\in \G : \|q(\g)\|\leq C\}\subseteq G\mathcal K H.$$
If $\mathcal G = \{\{e\}\} $, then this is the usual notion of metric properness, in which case the map $q$ itself is referred to as \textit{proper};
\item \textit{symmetric} if $\pi_\g (q(\g^{-1})) = q(\g)$ for all $\g\in \G$;
\item \textit{anti-symmetric} if $\pi_\g(q(\g^{-1})) = -q(\g)$ for all $\g\in \G$; and
\item \textit{uniform} if there exists a proper length function $\abs{\,\cdot\,}$ on $\G$ and an increasing function $\rho: \bb R_{\geq 0}\to \bb R_{\geq 0}$ such that $\rho(t)\to \infty$ as $t\to \infty$ and such that
\begin{equation*} \rho(\abs{\g^{-1}\de}) \leq \nor{q(\g) - q(\de)},
\end{equation*}
for all $\g,\de\in \G$.
\end{itemize}

\end{definition}

\begin{remark} In the preceding definition we could as well have relaxed the condition of strict (anti-)symmetry to merely the condition that $\nor{\pi_\g(q(\g^{-1})) \pm q(\g)}$ is bounded. However, it is easy to check that for any such function $q$, there exists an array $\tilde q$ which is a bounded distance from $q$; namely, $\tilde q(\g) = \frac{1}{2}(q(\g) \pm \pi_\g(q(\g^{-1})))$. This observation is essentially due to Andreas Thom \cite{Tho}.
\end{remark}

It is easy to see that a length function on a group is a uniform, symmetric array for the trivial representation. Our primary examples of (uniform) anti-symmetric arrays will be quasi-1-cocycles.

\begin{definition}\label{quasi-cocycle} Let $\G$ be a countable discrete group and $\pi: \G\ra  \Cal O( \Cal H_\pi)$ be an orthogonal representation of $\G$ on a real Hilbert space $ \Cal H_\pi$. A map $q: \G\ra \Cal H_\pi$ is called a \emph{quasi-1-cocycle} for the representation $\pi$ if one can find a constant $K\geq 0$ such that for all $\g,\la\in \Gamma$ we have\begin{equation}\label{1}\|q(\g\la)-q(\g)-\pi_\g(q(\la))\| \leq K.\end{equation}
\end{definition}

We denote by $D(q)$ the \textit{defect} of the quasi-$1$-cocycle $q$, which is the infimum of all $K$ satisfying equation (\ref{1}). Notice that
when the defect is zero the quasi-1-cocycle $q$ is actually a $1$-cocycle for $\pi$ \cite{BV}. In the sequel, we will drop the ``1'' and refer
to (quasi-)1-cocycles as (quasi-)cocycles. Again, without (much) loss of generality we will require a quasi-cocycle $q$ to be anti-symmetric,
since every quasi-cocycle $q$ is a bounded distance from some anti-symmetric quasi-cocycle $\tilde q$, cf. \cite{Tho}, which will suffice for
our purposes.

A distinct advantage to working with arrays is that, unlike cocycles (or even quasi-cocycles), there is a well-defined notion of a tensor product.

\begin{prop}\label{productarrays} Let $\G$ be a countable discrete group. Let $\pi_i:\G\ra \mathcal O(\mathcal H_i)$ be an orthogonal representation for $i = 1,2$, and let  $q_i:\G\ra \Cal  H_i$ be an array for $\pi_i$. Denote by \[\kappa(\g)=\max_{i = 1,2}\|q_i(\g)\|+1\] for all $\g\in \G$. Then the map $ q_1\wedge q_2:\G\ra \Cal  H_1\otimes \Cal  H_2$ defined by \begin{equation}q_1\wedge q_2(\g)=\kappa(\g)^{-1}q_1(\g)\otimes q_2(\g)\end{equation} is an array into the tensor representation $\pi_1\otimes \pi_2$. Moreover, if the arrays $q_i$ are assumed to be symmetric, then $q_1\wedge q_2$ is symmetric. If each of the arrays $q_i$ is assumed to be proper relative to a given family $\mathcal G_i$ of subgroups  of $\G$ then $q_1\wedge q_2$ is proper with respect to the family $\mathcal G_1\cup \mathcal G_2$.\end{prop}

\begin{proof} For brevity, we will denote $\pi_1\otimes\pi_2$ as $\pi$ and $q_1\wedge q_2$ as $q$. First, we show that if $q_1,q_2$ are arrays then one can find a function $r:\G \ra \mathbb R_+$ such that for all $ \g, \la\in\G$ we have  \begin{equation}\label{2002}|\kappa(\g\la)-\kappa (\la)|\leq r(\g).\end{equation}

Indeed, by applying the triangle inequality, we see that \begin{equation}\label{2009}\begin{split}\kappa( \g\la)&=\max\{\|q_1(\g\la)\|,\|q_2(\g\la)\|\}+1\\&\leq \max \{\|q_1(\la)\|,\|q_2(\la)\|\}+1+\max\{\|q_1(\g\la)-(\pi_1)_{\g}(q_1(\la))\|, \| q_2(\g\la)-(\pi_2)_{\g}(q_2(\la))\|\}\\
&\leq \kappa(\la)+\max\{\|q_1(\g\la)-(\pi_1)_{\g}(q_1(\la))\|, \| q_2(\g\la)-(\pi_2)_{\g}(q_2(\la))\|\}\end{split}\end{equation}

\noindent Since $q_1$ and $q_2$ are arrays, there exists a function $\g\mapsto r'(\g)$ such that $\max\{\|q_1(\g\la)-(\pi_1)_{\g}(q_1(\la))\|, \| q_2(\g\la)-(\pi_2)_{\g}(q_2(\la))\|\}\leq r'(\g)$ for all $\la\in\G$. Using this notation (\ref{2009}) can be rephrased as $\kappa(\g\la)\leq r'(\g)+\kappa(\la)$ for all $\g,\la\in \G$. This implies that $\kappa(\la)=\kappa(\g^{-1}\g\la)\leq r'(\g^{-1})+\kappa( \g \la)$ for all $\g,\la\in \G$. Therefore, when combining the last two inequalities we conclude that \begin{equation*}|\kappa(\g\la)-\kappa(\la)|\leq\max \{r'(\g),r'(\g^{-1})\},\end{equation*}
 \noindent for all $\g, \la\in \G$. Letting $r(\g)=\max \{r'(\g),r'(\g^{-1})\} $ we obtain (\ref{2002}).

 Next, we will show that $q$ is an array; that is, that $ q$ is boundedly equivariant. Applying the triangle inequality, we have the following estimates:

\begin{eqnarray*}&&\|q(\g\la)-\pi_\g(q(\la))\|\\&=&\|\kappa(\g\la)^{-1}q_1(\g\la)\otimes q_2(\g\la)-\kappa(\la)^{-1}(\pi_1)_{\g}(q_1(\la))\otimes (\pi_2)_{\g}(q_2(\la))\|\\
&\leq &\kappa( \g \la)^{-1}\|\left (q_1(\g\la)-(\pi_1)_{\g}(q_1(\la))\right )\otimes q_2(\g\la)\|\\&\quad&+ \left |\kappa(\g \la)^{-1}-\kappa(\la)^{-1} \right |\|\pi_{\g}(q_1(\la))\otimes q_2(\g\la)\|\\
&\quad&+\kappa(\la)^{-1}\|(\pi_1)_{\g}(q_1(\la))\otimes \left ( q_2(\g\la)-(\pi_2)_{\g}(q_2(\la))\right)\|\\
&\leq& \|q_1(\g\la)-(\pi_1)_{\g}(q_1(\la))\|+ \| q_2(\g\la)-(\pi_2)_{\g}(q_2(\la))\|+|\kappa(\g \la)-\kappa(\la)|\\
&\leq& \|q_1(\g\la)-(\pi_1)_{\g}(q_1(\la))\|+ \| q_2(\g\la)-(\pi_2)_{\g}(q_2(\la))\|+|r( \g)|.\end{eqnarray*}

Since $q_i$ is boundedly equivariant, $i=1,2$, the previous inequality combined with (\ref{2002}) shows that $q$ is boundedly equivariant.

From the definitions, one can easily see that if each $q_i$ is symmetric, then $q$ is again symmetric.

Finally, we verify the properness condition. Let $C>0$ be a fixed arbitrary constant and denote by $ K=\{\g\in \G : \| q(\g)\|\leq C\}$. A straightforward computation shows that if $\g\in  K$ then either $\|q_1(\g)\|\leq C$ or   $\|q_2(\g)\|\leq C$. Since each $q_i$ is proper relative to $ G_i$, so there exist finite sets $G_1, G_2, H_1, H_2 \subset \G$, $ K_1\subset \mathcal G_1$, and $ K_2\subset \mathcal G_2$ such that $ K \subseteq G_1  K_1H_1\cup G_2 K_2H_2$. Since this holds for all $C>0$, we have obtained that $q$ is proper relative to $\mathcal G_1\cup \mathcal G_2$.  \end{proof}

\begin{prop}\label{prop:symm-array} Let $\pi: \G\to \Cal  O( \Cal H_\pi)$ be an orthogonal representation. Assume that $\G$ admits a proper array $q: \G\to\Cal H_\pi$ for $\pi$ which is boundedly bi-equivariant, i.e., $ \| q(\g\de \la) - \pi_\g q(\de) \| < C(\g,\la), $
for all $ \g, \de, \la\in \G $. Here, $C(\g,\la)$ denotes a constant only depending on $\g$ and $\la$. Then there exists a symmetric proper array $\tilde q: \G\ra \mathcal H_\pi\otimes \mathcal H_\pi$ for the diagonal representation $\pi\otimes\pi$.
\end{prop}

\begin{proof}

We begin by observing that if $q$ is proper and boundedly bi-equivariant, then the map $q':\G\ra \mathcal H_\pi$ defined by  $ q'(\g) = \pi_\g (q(\g^{-1})) $ is also boundedly equivariant and obviously proper.

Indeed, to see this we note that be definition we have \begin{eqnarray*}\| q'(\g\de) -\pi_\g( q'(\de))\|=\|  q(\de^{-1}\g^{-1}) -  q(\de^{-1}) \|\leq C(e,\g^{-1}).\end{eqnarray*} Notice that the above constant depends only on $\g$.

Now we consider $\tilde q:\G\ra \mathcal H_\pi\otimes \mathcal H_\pi$ to be the symmetric product of the boundedly equivariant maps $\tilde q(\g) = \frac{1}{2} \left (q \wedge q' (\g) + q' \wedge q(\g) \right )$, which is also boundedly equivariant from the previous proposition. It is a straightforward exercise to check that $\tilde q$ is symmetric, i.e.\  $(\pi\otimes \pi)_\g (\tilde q(\g^{-1})) = \tilde q(\g)$.

Finally, since the square of the norm of the symmetric product of two vectors $x$ and $y$ is $ \|x\|^2\|y\|^2 + |\langle x,y\rangle |^2 $, we have that $q$ is proper implies that $\tilde q$ is also proper.

\end{proof}

\subsection{The classes $\Cal{QH}$ and $\Cal{QH}_{\rm reg}$}

We now proceed to describe some ``cohomological'' properties of countable discrete groups which capture many aspects of negative curvature from
the perspective of representation theory.

\begin{definition}\label{defn:QH} We say that a countable discrete group $\G$ is in the class $\Cal{QH}$ if it admits a proper, symmetric array $q: \G\to \Cal H_\pi$ for some non-amenable unitary representation $\pi:\G\to \Cal U( \Cal H_\pi)$. If the representation $\pi$ can be chosen to be weakly-$\ell^2$, then we say that $\G$ belongs to the class $\Cal{QH}_{\rm reg}$.
\end{definition}

By Proposition \ref{prop:symm-array} we see that the class $\Cal{QH}_{\rm reg}$ generalizes the class $\Cal D_{\rm reg}$ of Thom \cite{Tho} and that the class $\Cal{QH}$ contains all groups having Ozawa and Popa's property (HH) \cite{OPCartanII}.

\begin{prop}\label{prop:qh_permanence} The following statements are true.

\begin{enumerate}

\item If $\G_1,\G_2\in\Cal {QH}$, then so are $\G_1\times \G_2$ and $\G_1\ast \G_2$.

\item If $\G_1,\G_2\in\Cal {QH}_{\rm reg}$, then $\G_1\ast \G_2\in\Cal {QH}_{\rm reg}$.

\item If $\G_1,\G_2\in\Cal {QH}_{\rm reg}$ are non-amenable, then $\G_1\times \G_2\not\in\Cal {QH}_{\rm reg}$.

\item If $\G$ is a lattice in a simple connected Lie group with real rank one, then $\G\in\Cal {QH}_{\rm reg}$.

\item If $\G\in \Cal{QH}$, then $\G$ is not inner amenable. If in addition $\G$ is weakly amenable, then $\G$ has no infinite normal amenable subgroups.

\end{enumerate}
\end{prop}

Statement (5) is essentially Proposition 2.1 in \cite{OPCartanII} combined with Theorem A in \cite{OzCBAP}.

\begin{proof} Statements (1) and (2) follow exactly as they do for groups which admit a proper cocycle into some non-amenable (respectively, weakly $\ell^2$) unitary representation, cf. \cite{Tho}.

We prove statement (3) under the weaker assumption that $\G\cong \La\times \Sg$, where $\La$ is non-amenable and $\Sg$ is an arbitrary infinite
group. By contradiction, assume $\G$ admits a proper, symmetric, boundedly equivariant map $q : \G\to  \ell^2(\G)$ (by inspection, the
same argument will hold for $q : \G\to \Cal H_\pi$ for any weakly-$\ell^2$ unitary representation $\pi$). Since the action of $\La$ on
$\ell^2(\G)$ has spectral gap and admits no non-zero invariant vectors, there exists a finite, symmetric subset $S\subset \La$ and $K' >0$ such
that
\begin{equation*} \nor{\xi}\leq K'\sum_{s\in S}\nor{\la_s(\xi) - \xi},
\end{equation*}
for all $\xi\in \ell^2(\G)$. Let $K''\geq 0$ be a constant so that inequality (\ref{weakquasi}) is satisfied for $S\subset \G$, and set $K =
\max\{K',K''\}$. We then have for any $g\in \Sg$ that
\begin{equation}\label{eq:sg_bound}\begin{split}
\nor{q(g)} & \leq K\sum_{s\in S}\nor{\la_s(q(g)) - q(g)} \\
& \leq K\sum_{s\in S}\nor{q(sg) - q(g)} + K^2\abs{S} \\
& = K\sum_{s\in S}\nor{q(gs) - q(g)} + K^2\abs{S} \\
& = K\sum_{s\in S}\nor{\la_{gs}(q(s^{-1}g^{-1})) - \la_g(q(g^{-1}))} + K^2\abs{S} \\
& = K\sum_{s\in S}\nor{\la_{s^{-1}}(q(g^{-1})) - q(s^{-1}g^{-1})} + K^2\abs{S}\leq 2K^2\abs{S}.
\end{split}
\end{equation}
Hence, $\nor{q(g)}$ is bounded on $\Sg$, which contradicts that $q$ is proper.

For statement (4), it is well known that any co-compact lattice in a simple Lie group with real rank one is Gromov hyperbolic; hence, by
\cite{MMS} it admits a proper quasi-cocycle into the left-regular representation. A result of Shalom, Theorem 3.7 in \cite{Shalom}, shows that any
lattice in such a Lie group is integrable, and therefore $\ell^1$-measure equivalent to any other lattice in the same Lie group. It is easy to
check that having a proper quasi-cocycle into the left-regular representation is invariant under $\ell^1$-measure equivalence, cf. Theorem 5.10
in \cite{Tho}.

In order to prove statement (5), we assume by contradiction that $\G$ is inner amenable, i.e., there exists a state $\vp$ on $\ell^\infty(\G)$
such that $\vp \perp \ell^1(\G)$ and $\vp\circ{\rm Ad}(\g) = \vp$ for all $\g\in \G$. Let $q: \G\to \Cal H_\pi$ be an array into a non-amenable
representation $\pi$. Define a u.c.p.\ map $T: \fr B(\Cal H_\pi)\to \ell^\infty(\G)$ by $T(x)(\g) = \frac{1}{\nor{q(\g)}^2}\ip{x
q(\g)}{q(\g)}$. Similarly to the proof of statement (3), by symmetry and bounded equivariance, for every $\g\in \G$, there exists $K\geq
0$ such that
\begin{equation} \nor{q(\g^{-1}\de\g) - \pi_{\g^{-1}}(q(\de))}\leq K,
\end{equation}
for all $\de\in \G$. Since $q$ is proper, this implies that the state $\Phi = \vp\circ T$ on $\fr B(\Cal H_\pi)$ is ${\rm Ad}(\pi)$-invariant.
However, this contradicts the fact that $\pi$ is a non-amenable representation. The remaining assertion follows by Theorem A in \cite{OzCBAP}.

\end{proof}

The class $\Cal{QH}_{\rm reg}$ is intimately related with Ozawa's class of bi-exact groups (often denoted as the class $ \Cal S$ in the literature, e.g., \cite{OzKurosh, Sako}). A reader unfamiliar with the theory of exact groups should consult Appendix \ref{sec:biexact} before proceeding further.

\begin{definition}[Ozawa \cite{BrOz,OzKurosh}]\label{defn:starA} A countable discrete group $\G$ is said to be \textit{bi-exact} if it admits a sequence $\xi_n : \beta'\G\to \ell^2(\G)$ of continuous maps such that $\xi_n\geq 0$, $\nor{\xi_n(x)}_2 = 1$, for all $x\in \beta'\G$, $n\in \bb N$, which satisfy
\begin{equation}\label{eq:starA} \sup_{x\in \beta'\G}\nor{\la_\g(\xi_n(x)) - \xi_n(\g x\de)}_2\to 0,
\end{equation}
for all $\g,\de\in \G$. Here $\beta'\G=\beta\Gamma\setminus \Gamma$ denotes the Stone--\v{C}ech boundary.\end{definition}
It is easy to see that if $\G$ is bi-exact in the sense of Definition 15.1.2 of \cite{BrOz} if and only if $\G$ is bi-exact in the sense of the
above definition. By the same proof that ``property A $\Rightarrow$ coarse embeddability into Hilbert space'' (cf.\ \cite{BrOz, Roe}), we have the following

\begin{prop}\label{prop:starA} If $\G$ is bi-exact, then it admits a uniform array into $\ell^2(\G)^{\oplus\infty}$. In particular, $\G$ is exact and belongs to the class $\Cal{QH}_{\rm reg}$.
\end{prop}

Indeed, from the maps $\xi_n$, one may construct a proper, boundedly bi-equivariant map into $\ell^2(\G)^{\oplus\infty}$, which may in turn be used to construct a proper, symmetric array by Proposition \ref{prop:symm-array}.

\begin{remark} After a preliminary version of this manuscript was circulated, Narutaka Ozawa pointed out that the converse is also true. That is, if a countable discrete group $\G$ is exact and belongs to the class $\Cal{QH}_{\rm reg}$, then $\G$ is bi-exact. A proof for the special case of the left-regular representation is contained in \cite{CSU}: the general case may be found in \cite{PVhypcartan}.
\end{remark}

The class $\Cal{QH}_{\rm reg}$ is strictly larger than the class $ \Cal D_{\rm reg}$ of Thom \cite{Tho}. This follows from Ozawa's proof that the group $\bb Z^2\rtimes {\rm SL}(2,\bb Z)$ is bi-exact \cite{OzExample}, in conjunction with a theorem of Burger and Monod \cite{BuMo} demonstrating that $\bb Z^2\rtimes {\rm SL}(2,\bb Z)$ admits no proper quasi-cocycle for any representation. However, it is instructive to supply a direct proof without appealing to bi-exactness.

\begin{prop}\label{prop:SL} The group $\bb Z^2\rtimes {\rm SL}(2,\bb Z)$ is in the class $\Cal{QH}_{\rm reg}$.
\end{prop}

\noindent The details of the construction are found in Appendix \ref{sec:SL}.

\section{Deformations of the uniform Roe algebra}

\subsection{Schur multipliers and the uniform Roe algebra.}\label{sec:Roe} Using exponentiation, we now describe a canonical way to associate to an array $q:\G\ra \Cal H_\pi$ a family of multipliers $\mathfrak
m_t$ on $\mathfrak B(\ell^2(\G) )$. First notice that the kernel $(\g,\de)\mapsto \|q(\g)-q(\de)\|^2$ is conditionally negative definite (cf.
Section 11.2 in \cite{Roe} or Appendix D in \cite{BrOz}) and therefore by Schoenberg's theorem \cite{Roe}, for every $t\in \mathbb R$, the
kernel $\kappa_t(\g,\de)=\exp(-t^2\|q(\g)-q(\de)\|^2)$ is positive definite. Hence for every $t$ there is a unique unital, completely positive
(u.c.p.) map $\mathfrak m_t:\mathfrak B(\ell^2(\Gamma))\rightarrow \mathfrak B(\ell^2(\Gamma))$ called a \emph{Schur multiplier}, such that
\begin{equation}\mathfrak m_t([x_{\g,\de}])=[\kappa_t(\g,\de)x_{\g,\de}],\end{equation} for all $x\in \mathfrak B(\ell^2(\G)).$

If $\G$ is a group then the \emph{uniform Roe algebra} $C^*_u(\G)$ is defined as the $\rm C^*$-subalgebra of $\mathfrak B(\ell^2(\G))$
generated by $C^*_\la(\G)$ and $\ell^{\infty}(\G)$. Notice that if one considers the action $\G\curvearrowright^{\la} \ell^{\infty}(\G)$ by
left translation, then the uniform algebra $C^*_u(\G)$ can be canonically identified with the reduced crossed product $\rm C^*$-algebra
$\ell^{\infty}(\G)\rtimes_{\la,r}\G$. Let $ F_0$ denote the net of unital, symmetric, finite subsets of $\G$. Given a finite subset $F \in
 F_0$, we define the operator space of \textit{$F$-width operators} $X(F)$ to be the space of bounded operators $x \in \mathfrak
B(\ell^2(\G))$ such that $x_{\g,\de} = 0$ whenever $\g^{-1}\de \in \G \setminus F$. Since it is easy to check that $X(F) = \ell^\infty(\G)\bb
C[F]\ell^\infty(\G)$, we have ${\rm C}_u^\ast(\G) = \overline{span\{X(F) : F \in  F_0\}}^{\|\,\cdot\,\|}$. Our interest in the uniform Roe
algebra stems from the fact that it is, in practical terms, the smallest ${\rm C}^\ast$-algebra which contains ${\rm C}_\la^\ast(\G)$ and which
is invariant under the class of Schur multipliers associated to $\G$.

\begin{prop} The algebra $C^*_u(\G)$ is invariant under $\mathfrak m_t$.
\end{prop}
\begin{proof} Let $x \in {\rm C}_u^\ast(\G)$, then there exists a sequence $(x_n)$ of elements of $span\{X(F) : F \in  F_0\}$ such that $\nor{x_n - x}_\infty\to 0$. It is easy to see that if $x_n$ is supported on the set $F_n \in  F_0$, then $\mathfrak m_t(x_n) \in X(F_n)$. Since $\nor{\mathfrak m_t(x_n) - \mathfrak m_t(x)}_\infty\to 0$, we have that $\mathfrak m_t(x) \in {\rm C}_u^\ast(\G)$.
\end{proof}

We will also heavily use the following observation of Roe on the convergence of $\fr m_t$ on ${\rm C}_u^\ast(\G)$, cf. Lemma 4.27 in
\cite{Roe}.

\begin{prop}\label{Roe's_Lemma} If $q$ is an array, then for all $x \in {\rm C}_u^\ast(\G)$, we have that $\nor{\fr m_t(x) - x}_\infty\to 0$ as $t\to 0$.
\end{prop}

\subsection{Construction of the extended Roe algebra $C^*_{u}(\G\ca^{\rho} Z)$.}\label{construction} Let $\G\ca^{\rho} (Z,\eta)$ be a measure preserving action on a probability space
$Z$. By abuse of notation, we still denote by $\rho$ the Koopman representation of $\G$ on $L^2(Z,\eta)$ induced by the action $\rho$. Then consider
the Hilbert space $L^2(Z,\eta)\otimes \ell^2(\G)$ and for every $\g\in\G$ define a unitary $u_\g\in \mathfrak B(L^2(Z)\otimes \ell^2(\G))$ by the
formula
\begin{equation*}u_\g(\xi\otimes \delta_h)=\rho_\g(\xi)\otimes \delta_{\g h},\end{equation*}
where $\xi\in L^2(Z)$ and $h\in  \G$. Consider the algebra $L^\infty(Z \times \G,\eta \times c)\subset \mathfrak B(L^2(Z)\otimes \ell^2(\G))$,
where $c$ is the counting measure on $\G$. Then the \emph{extended Roe algebra} $C^*_{u}(\G\ca^{\s} Z)$ is defined as the $C^*$-algebra
generated by $L^\infty(Z\times \G)$ and the unitaries $u_\g$ inside $\mathfrak B(L^2(Z)\otimes \ell^2(\G))$. Notice that when $X$ consists of a point, our definition recovers the regular uniform Roe algebra, i.e., $C^*_{u}(\G\ca^{\rho} Z)=C^*_{u}(\G)$.

As in the case of the uniform Roe algebra, we will see that $C^*_{u}(\G\ca^{\rho} Z)$ can be realized as a reduced crossed product algebra. Specifically, we consider the action $\G\curvearrowright L^\infty(Z\times \G)$ given by $\la^\rho_\g(f)(x,h)=f(\g^{-1}x,\g^{-1}h)$, where $f\in L^\infty(Z\times \G)$, $x\in Z$ and $\g,h\in \G$. Then we show that $C^*_{u}(\G\ca^{\rho} Z)$ is naturally identified  with the reduced crossed product algebra corresponding to this action and the faithful representation $L^\infty(Z\times\G)\subset \mathfrak B(L^2(Z)\otimes \ell^2(\G))$.
\begin{prop}$C^*_{u}(\G\ca^{\rho} Z)\cong L^{\infty}(Z\times \G)\rtimes_{\la^\rho, r} \G$. \end{prop}
\begin{proof} Consider the operator $U:L^2(Z)\otimes \ell^2(\G)\otimes \ell^2(\G)\ra L^2(Z)\otimes \ell^2(\G)\otimes \ell^2(\G)$ defined by $U(\xi\otimes \de_k\otimes \de_h)=\s_h (\xi)\otimes \de_k\otimes \de_{h k}$, where $\xi\in L^2(Z)$ and $\g,h\in\G$. One can easily check this is a unitary, and below we will show it implements a spatial isomorphism between the two algebras. For this purpose we will be seeing $C^*_u(\G\curvearrowright Z)$ as the $C^*$-algebra generated by $L^\infty(Z\times\G)$ and $u_\g$ inside $\mathfrak B (L^2(Z)\otimes \ell^2(\G)\otimes \ell^2(\G))$, where
\begin{equation}\label{40}\begin{split}
f(\xi\otimes \de_k \otimes \de_h)&=(f(\,\cdot \, ,h)\xi)\otimes \de_k\otimes \de_h\\
u_\g(\xi\otimes\de_k\otimes \de_h)&=\rho_\g(\xi)\otimes \de_k\otimes \de_{\g h},
\end{split}\end{equation}
for all $f\in L^\infty(Z\times\G)$, $\g,h, k\in \G$, and $\xi\in L^2(Z)$.

Using the formula for $U$ in combination with equations (\ref{40}), we have
\begin{equation*}\begin{split} U(1\otimes \la_\g)(\xi\otimes \de_k\otimes \de_h)&=U(\xi\otimes \de_k\otimes \de_{\g h})\\ &=\rho_{\g h}(\xi)\otimes \de_k\otimes \de_{\g hk}\\&=u_\g(\rho_{h}(\xi)\otimes \de_k\otimes \de_{h k})\\&=u_\g U(\xi\otimes \de_k\otimes \de_{h}); \end{split}\end{equation*} hence, $U(1\otimes \la_\g)U^*=u_\g$ for all $\g \in\G$.

We also consider the representation of $L^\infty(Z\times\G)$ on $\mathfrak B(L^2(Z)\otimes \ell^2(\G)\otimes \ell^2(\G))$ given by
$\pi(f)(\xi\otimes \de_k\otimes \de_h)= \la^\rho_{h^{-1}}(f)(\xi\otimes\de_k)\otimes \de_h$, for every $f\in L^\infty(Z\times \G)$.

\noindent Combining this with equations (\ref{40}) we see that \begin{equation*}\begin{split} U\pi( f) (\xi\otimes \de_k\otimes
\de_h)&=U((\la^\rho_{h^{-1}}(f)(\xi\otimes \de_k)\otimes \de_{h})\\ &= U(\rho_{h^{-1}}(f(\, \cdot \,, hk))\xi\otimes \de_k\otimes \de_{h})\\&=f(\,
\cdot\, , h k)\rho_h(\xi)\otimes \de_k\otimes \de_{h\g}\\&= f U(\xi\otimes \de_k\otimes \de_{h}). \end{split}\end{equation*} Therefore, for all
$f\in L^\infty(Z\times \G)$ we have $U\pi(f)U^*=f$, and from the discussion above we conclude that $U(L^{\infty}(Z\times \G)\rtimes_{\la^\rho, r}
\G )U^* =C^*_u(\G\curvearrowright Z )$.\end{proof}

For further reference we keep in mind the following diagram of canonical inclusions:
\begin{eqnarray}\label{7}
\begin{array}{ccccc}
  L^\infty(Z)\rtimes_{\rho,r} \G & \subset & L^\infty(Z\times \G)\rtimes_{\la^\rho,r}\G&=&C^*_{u}(\G\ca^{\rho}Z)  \\
  \cup &  & \cup & &\\
  C^*_r(\G) & \subset & \ell^{\infty}(\G)\rtimes_{\la,r}\G & =&C^*_{u}(\G)
\end{array}
\end{eqnarray}
Note there exists a conditional expectation $E:L^\infty(Z\times \G) \ra  \ell^\infty(\G)$ defined by $E(f)(\g)=\int_Z f(x,\g)d\mu(x)$. This map
is clearly $\G$-equivariant and thus it extends to a conditional expectation $\tilde E: C^*_u(\G\curvearrowright Z) \ra C^*_u( \G)$ by letting
$\tilde E(\sum_\g x_\g u_\g)=\sum_\g E(x_\g)u_\g$ for any $\sum_\g x_\g u_\g \in C^*_u(\G\curvearrowright Z) $ with $x_\g\in L^\infty(Z\times
\G)$.

\subsection{A path of automorphisms of the extended Roe algebra associated with the Gaussian action}\label{path} Let $\G\ca^\s (X,\mu)$ be a measure preserving action of $\G$ on a probability space $X$.  Following \cite{PeSi}, any orthogonal group representation $\pi: \G \to \mathcal O ( \Cal H_\pi)$, gives rise to a measure-preserving action which we still denote by $\G \ca^\pi (Y^\pi,\nu^\pi)$, called the \emph{Gaussian action}. We consider the diagonal action $\G\ca^{\s\otimes \pi} (X\times Y^\pi,\mu\times \nu^\pi)$ and which will be denoted by $\G \ca^{\rho} (Z,\zeta)$. As in the previous subsection, to this action we associate the extended Roe algebra $C^*_{u}(\G\ca^{\rho} Z)$. Below we indicate a procedure to construct a one-parameter family $(\a_t)_{t \in \bb R}$ of $\ast$-automorphisms
of $C^*_{u}(\G\ca^{\rho} Z)$. Specifically, $\a_t$ is obtained by exponentiating an array $q:\G\ra \mathcal H_\pi$ in a similar way to the
construction of the malleable deformation of $L\G$ from a cocycle $b$ as carried out in \S 3 of \cite{Sin}. Crucially, this family will be
continuous with respect to the uniform norm as $t \to 0$ (Lemma \ref{deform-ineq}).

Following the construction presented in \S 1.2 of \cite{Sin}, given an array $q: \G \to \Cal H_\pi$, there exists a one-parameter family of maps $\up_t: \G \to \Cal U(L^\infty(Y^\pi,\nu^\pi))$ defined by $\up_t(\g)(x)=\exp(it q(\g)(x))$, where $\g\in \G$, $x\in Y^\pi$. Using similar computations as in \cite{PeSi,Sin} one can verify that we have the following properties:
\begin{prop}
\begin{eqnarray}
\label{19}&&\text{If }\pi\text { is weakly-}\ell^2\text{ then the Koopman representation }{\pi_\s}_{|L_0^2(Y^\pi,\nu^\pi)}\\&& \nonumber\text{ is also weakly-}\ell^2;\\
\label{13} &&\int \up_t(\g)(x) \up_t(\de)^*(x) d\mu(x) = \kappa_t(\g,\de) \text{ for all }\g,\de\in \G.
\end{eqnarray}
\end{prop}

These maps give rise naturally to a path of operators $V_t\in\mathfrak B( L^2(Y^\pi)\otimes L^2(X)\otimes \ell^2(\G))$ by letting
$V_t(\xi\otimes \eta\otimes \delta_\g)=(\upsilon_t(\g)\xi)\otimes\eta\otimes  \delta_\g$, for every $\xi\in L^2(Y^\pi)$, $\eta\in L^2(X)$ and $\g
\in \G$. For further reference we summarize below some basic properties of $V_t$.

\begin{prop}\label{45} For every $t,s\in\mathbb R$ we have the following properties:
\begin{enumerate}
\item $V_tV_s=V_{t+s}$, $V_tV^*_t=V^*_tV_t=1$;
\item If the array is anti-symmetric we have $JV_tJ=V_{t}$ and if it is symmetric we have $JV_tJ=V_{-t}$. Here we denoted by  $J: L^2(L^\infty (Z)\rtimes\G)\ra L^2(L^\infty (Z)\rtimes \G)$ Tomita's conjugation.
\end{enumerate}\end{prop}

\begin{proof} The first part follows directly from the definitions, so we leave the details to the reader.
To get the second part it suffices to verify that the two operators coincide on vectors of the form $\xi\otimes \eta\otimes\delta_\g\in
L^2(Y^\pi)\otimes L^2(X)\otimes \ell^2(\G)$. If we assume that $q$ is an anti-symmetric array then employing the formulas for $J$, $V_t$, $\upsilon_t$ we see that
\begin{equation*}\label{25}\begin{split}JV_tJ(\xi\otimes \eta\otimes \delta_\g)&=JV_t((\sigma_{\g^{-1}}(\xi^*))\otimes \s_{\g^{-1}}(\eta^*)\otimes \delta_{\g^{-1}})\\
&=J(\upsilon_t(\g^{-1})\sigma_{\g^{-1}}(\xi^*)\otimes \s_{\g^{-1}}(\eta^*)\otimes \delta_{\g^{-1}})\\
&=(\sigma_\g(\upsilon_{-t}(\g^{-1}))\xi)\otimes \eta\otimes \delta_{\g}\\
&= (\exp(-it \pi_\g(q(\g^{-1})))\xi)\otimes \eta\otimes \delta_\g \\
&=(\exp(it q(\g))\xi)\otimes \eta\otimes \delta_\g \\&=V_{t}(\xi\otimes \eta\otimes \delta_\g),\end{split}\end{equation*} which finishes the
proof in this case.

When the array is symmetric we get the conclusion by a similar computation. In this case the details are left to the reader. \end{proof} Since $V_t$ is a unitary on $L^2(Z)\otimes \ell^2(\G)$, we may consider an inner automorphism $\a_t$ of $\mathfrak B
(L^2(Z)\otimes \ell^2(\G))$ by letting $\a_t(x)= V_t xV^*_t$ for all $x\in\mathfrak B (L^2(Z)\otimes \ell^2(\G))$. Notice that this formula gives a family of inner automorphisms of the extended Roe algebra. Moreover, when restricting to the extended Roe algebra $C^*_u(\G\ca X)$ one can recover from $\a_t$ the multipliers introduced above: $E\circ\a_t(x)= \id_X\otimes \mathfrak m_t(x)$ for all $x\in C^*_u(\G\ca X)$.

However, one can see right away that these automorphisms do not move the group-measure space von Neumann algebra $L^\infty(X)\rtimes \G$ into itself. Hence, applying the deformation/rigidity arguments at the level of von Neumann algebra $L^\infty(X)\rtimes \G$ is rather inadequate. As we will see in the next section, this difficulty is overcome by working with the reduced crossed product $C^*$-algebra $L^\infty (X)\rtimes_{\s,r}\G$ rather than $L^\infty(X)\rtimes \G$. The following result underlines that the path $\a_t$ is a deformation at the $ C^*$-algebraic level, i.e., with respect to the operatorial norm.

\begin{lem}\label{deform-ineq} Let $q$ be any symmetric or anti-symmetric array. Assuming the notations above, for every $x\in L^\infty (X)\rtimes_{\s,r}\G$ we have
\begin{eqnarray}
\label{3}&&\|(\a_t(x)-x)\cdot e\|_{\infty}\ra 0\text{ as }t\ra 0;\\
\label{4}&&\|(\a_t(JxJ)-JxJ)\cdot e\|_{\infty}\ra 0\text{ as }t\ra 0,\end{eqnarray} where $\nor{\, \cdot \,}_{\infty}$ denotes the operatorial
norm in $\mathfrak B(L^2(Z)\otimes \ell^2(\G))$. Here $e$ denotes the orthogonal projection from $L^2(Z)\otimes \ell^2(\G)$ onto $L^2(X)\otimes \ell^2(\G)$.
\end{lem}
\begin{proof}
Since elements in $L^\infty (X)\rtimes_{\s,r}\G$ can be approximated in the uniform norm by $\G$-finitely supported elements, using the
triangle inequality it suffices to show (\ref{3}) only for $x=\sum_{g\in F} x_g u_g$, a finite sum where $x_g\in L^\infty(X)$. Fix an arbitrary vector
$\xi=\sum_\g \xi_\g\otimes \delta_\g\in  L^2(X)\otimes \ell^2(\G )$. Using the formula for $\a_t$ in combination with the Cauchy-Schwarz
inequality, we have
\begin{equation}\label{10}\begin{split}
\|(\a_t(x)-x)\xi\|^2&= \|\sum_{g\in F}\sum_{\g\in\G} x_g(V_tu_g V_{-t}-u_g)(\xi_\g\otimes \delta_\g)\|^2\\
&\leq\left (|F|\max_{g\in F}\|x_g\|^2_\infty\right )\left (\sum_{g\in F} \|\sum_{\g\in\G}(V_tu_g V_{-t}-u_g)(\xi_\g\otimes \delta_\g)\|^2\right ) \\
\end{split}\end{equation} Applying the definitions and the formula for $V_t$ we see that $u_g(\xi_\g\otimes
\delta_\g)=\sigma_g(\xi_\g)\otimes \delta_{g \g}$ and $V_tu_g V_{-t}(\xi_\g\otimes \delta_\g)=\upsilon_t(g
\g)\sigma_g(\upsilon_{-t}(\g))\otimes \s_g(\xi_\g)\otimes\delta_{g \g}$. Therefore, continuing the estimate (\ref{10}) we obtain
\begin{equation}\label{11}\begin{split}
&=\left (|F|\max_{g\in F}\|x_g\|^2_\infty\right )\sum_{g\in F} \|\sum_{\g\in\G}(\upsilon_t(g \g)\sigma_g(\upsilon_{-t}(\g))-1)\otimes \s_g(\xi_\g)\otimes \delta_{g \g}\|^2\\
&=\left (|F|\max_{g\in F}\|x_g\|^2_\infty\right )\sum_{g\in F} \sum_{\g\in\G}\|\xi_\g\|^2\|\upsilon_t(g \g)\sigma_g(\upsilon_{-t}(\g))-1\|^2\\
&=2\left (|F|\max_{g\in F}\|x_g\|^2_\infty\right )\sum_{g\in F} \sum_{\g\in\G}\|\xi_\g\|^2\left(1-\tau(\upsilon_t(g
\g)\sigma_g(\upsilon_{-t}(\g)))\right).
\end{split}
\end{equation}
On the other hand, the same computations as in the proof of (\ref{13}) together with inequality (\ref{weakquasi}) imply that, there exist
$K\geq0$ such that for every $g\in F$ and $\g \in \G$ we have
\begin{equation}\label{12}\begin{split}
\tau(\upsilon_t(g \g)\sigma_g(\upsilon_{-t}(\g)))&=\int_X\exp\left(it (q(g \g)-\pi_g(q(\g)))(x)\right )d\mu(x)\\
&=\exp\left(-t^2 \|q(g \g)-\pi_g(q(\g))\|^2\right)\\
&\ge\exp\left(-t^2 K\right).
\end{split}
\end{equation}
Thus, combining (\ref{10}), (\ref{11}) and (\ref{12}) we conclude that, for all $\xi\in L^2(X)\otimes \ell^2(\G)$, we have
\begin{equation*}\begin{split}
\|(\a_t(x)-x)\xi\|^2&\le 2\left (\max_{g\in F} \|x_g\|_\infty\right )\|\xi\|^2 |F|^2\left (1-\exp\left(-t^2 K\right)\right),
\end{split}
\end{equation*}
which further implies
 \begin{equation}\label{24}\begin{split} \|(\a_t(x)-x)\cdot  e\|_\infty&\le 2\left (\max_{g\in F} \|x_g\|_\infty\right )|F|^2\left (1-\exp(-t^2 K )\right ).
\end{split}\end{equation}Since $F$ is finite, then $\exp(-t^2 K)\ra 1$ as $t\ra 0$, and (\ref{3}) follows from (\ref{24}).

It remains to show ($\ref{4}$). Assume first that $q$ is anti-symmetric. Since $[e,J]=0$, by Proposition \ref{45} we see that  $(\a_t(JxJ)-JxJ)\cdot e =J((\a_t(x)-x)\cdot e )J$. Therefore, (\ref{4}) follows from (\ref{3}) because $J$ is an anti-linear isometry. Similarly, when the array is symmetric, by the second part of Proposition \ref{45}, we have $(\a_t(JxJ)-JxJ)\cdot e =J((\a_{-t}(x)-x)\cdot e )J$ and the conclusion follows by the same reasoning. \end{proof}

\begin{remark}For future reference we make the following observation: in the proof of the previous proposition we used the symmetry or anti-symmetry of the array to show that  the deformation converges pointwise in the operatorial norm to the identity on $JC_\la^*(\G)J$. Except for establishing this fact, there is no other instance where symmetry or anti-symmetry will be used in the sequel.
\end{remark}

Next we show that the path of unitaries $V_t$ satisfies a ``transversality'' property very similar to Lemma 2.1 in \cite{PoSG}. Our proof
follows closely the proof of Lemma 3.1 in \cite{Vae}: we include it here only for the sake of completeness.
\begin{lem}\label{transversality} If $V_t$ is the unitary defined above, then for all $\xi\in   L^2(X)\otimes \ell^2(\G)$ and all $t\in\mathbb R$ we have
\begin{equation}\label{20}2\nor{V_t(\xi)-e\cdot V_t(\xi)\|^2\ge \|\xi-V_t(\xi)}^2.
\end{equation}
\end{lem}

\begin{proof} Fix $\xi\in  L^2(X)\otimes \ell^2(\G)$ and assume that it can be written as $\xi=\sum_\g \xi_\g\otimes \delta_\g$ with $\xi_\g\in L^2(X)$.
Straightforward computations show that $V_t(\xi)=\sum_\g \upsilon_t(\g)\otimes \xi_\g\otimes\delta_\g$ and $e\cdot  V_t(\xi)=\sum_\g
\tau(\upsilon_t(\g))1\otimes \xi_\g \otimes\delta_\g$; thus, the left side of (\ref{20}) is equal to
\begin{equation}\label{16}\begin{split}
2\|V_t(\xi)-e\cdot  V_t(\xi)\|^2&=2\left(\|V_t(\xi)\|^2-\|e\cdot  V_t(\xi)\|^2\right )\\
&=2 \sum_\g \|\xi_\g\|^2\|\upsilon_t(\g)\|^2-\|\xi_\g\|^2|\tau(\upsilon_t(\g))|^2\\
&=2\sum_\g \|\xi_\g\|^2\left (1-|\tau(\upsilon_t(\g))|^2\right ).
\end{split}
\end{equation}
Applying the same formulas as above, we see that the right side of (\ref{20}) is equal to
\begin{equation}\label{23}\begin{split}
\|\xi-V_t(\xi)\|^2&= \|\xi\|^2+\|V_t(\xi)\|^2-2Re\langle V_t(\xi),\xi\rangle\\
&= 2\left (\|\xi\|^2-Re\langle V_t(\xi),\xi\rangle\right )\\
&= 2\sum_\g\|\xi_\g\|^2\left(1-\tau(\upsilon_t(\g))\right).
\end{split}
\end{equation}
Since we have $\tau(\upsilon_t(\g))=\exp(-t^2\|q(\g)\|^2)\ge\exp(-2t^2\|q(\g)\|^2)= |\tau(\upsilon_t(\g))|^2$, the conclusion follows from
(\ref{16}) and (\ref{23}).
\end{proof}

The multipliers $\mathfrak m_t$ arising from a proper quasi-cocycle behave in some sense as compact operators on $L\G$, i.e., $\mathfrak m_t$ is continuous from the weak operator topology to the strong operator topology. Results of this type will be used crucially in the proof of Theorem 4.1.

\begin{prop}\label{improuvconv2} Let $\fr m_t$ be the Schur multiplier associated to some proper quasi-cocycle $q$ on $\G$. If $v_k\in M=L \G$ is a bounded sequence of elements such that  $v_k \ra  0$ weakly, as $k\ra \infty$, then for every $t>0$ and every finite set $F\subset \G$ we have that
\begin{equation*}\lim_{k\ra \infty}\left (\sup_{\|\xi\|\leq 1}\|(\mathfrak m_t(v_k)\otimes 1)(P_F\otimes 1)\xi\|\right )=0,\end{equation*}
where $\xi\in L^2(M)\otimes L^2(M)$. Here we denoted by $P_F$ the orthogonal projection from $L^2(M)$ onto the linear span of the set $\{ \de_\a \,|\, \a\in F\}$.\end{prop}

\begin{proof} Let $\xi=\sum_s \xi_s\otimes \eta_s$ where $\{\eta_s\}_{s\in \mathbb N} $ is an orthonormal basis of $L^2(M)$ and $\xi_s=\sum_r \xi^s_r \de_r$. Notice that $\|\xi\|\leq 1$ amounts to $\sum_s\sum_r|\xi^s_r|^2\leq 1$. Without loss of generality we may assume that the set $F$ consists of one element, i.e., $F=\{\g\}$. Let $v_k =\sum_h v^k_hu_h$ be the Fourier expansion, where $v^k_h\in \mathbb C$. Then applying the formula for $\mathfrak m_t$ and using $\sum_s|\xi^s_\g|^2\leq 1$  we have that

\begin{equation}\label{551}\begin{split}\|(\mathfrak m_t(v_k)\otimes 1)(P_F\otimes 1)\xi\|^2&=\sum_s\| eV_t v_kV_{-t}(\xi^s_\g \de_\g)\|^2\\&=\sum_s\|\sum_h eV_tv^k_hu_hV_{-t}(\xi^s_\g \de_\g)\|^2\\
&=\sum_s\|\sum_h(\tau(\upsilon_t(h\g)\sigma_h (\upsilon_{-t}(\g)))1) (v^k_h \xi^s_\g)\de_{h\g}\|^2\\
&=\sum_s\sum_h|v^k_h\xi^s_\g|^2 |\tau(\upsilon_t(h\g)\sigma_h (\upsilon_{-t}(\g)))|^2\\
&=\sum_s\sum_h|v^k_h\xi^s_\g|^2 \exp\left(-t^2\|q(h\g)-\pi_h (q(\g))\|^2\right)\\
&\leq \sum_h|v^k_h|^2 \exp\left(-t^2\|q(h\g)-\pi_h (q(\g))\|^2\right).
\end{split}\end{equation}

\noindent Furthermore, by the quasi-cocycle relation, the last term above is smaller than
  \begin{equation}\label{10001}\leq \sum_h |v^k_h|^2 \exp\left(-\frac{t^2}{2}\|q(h)\|^2+t^2D(q)\right).\end{equation}

Fix $\varepsilon>0$. Since $q$ is proper there exists a finite subset $F_\varepsilon\subset \G$ such that
 $\displaystyle\frac{2}{t^2}\ln\left(\frac{2\exp \left(t^2 D(q)\right)}{\varepsilon}\right) \leq \|q(h)\|^2$, for all $h\in \G\setminus F_\varepsilon$. This obviously implies
 that \begin{equation}\label{561}\exp\left(-\frac{t^2}{2}\|q(h)\|^2+t^2D(q)\right)\leq \frac{\varepsilon}{2\max_k\|v_k\|^2_\infty},\end{equation} for all $h\in \G\setminus F_\varepsilon$. Since  the sequence $v_k$ converges weakly to $0$ as $k$ approaches $\infty$ and $F_\varepsilon$ is finite, one can find $k_\varepsilon\in\mathbb N$
 such that, for all $k \geq k_\varepsilon$ and all $h\in F_\varepsilon$, we have
 \begin{equation}\label{571}|v^k_h|\leq \left(\frac{\varepsilon}{2|F_{\varepsilon}|\max_{h\in F_\varepsilon}\exp\left(-\frac{t^2}{2}\|q(h)\|^2+t^2D(q)\right)}\right)^{\frac{1}{2}}.\end{equation}

Using (\ref{551}), (\ref{10001}), (\ref{561}), and (\ref{571}) we obtain that for all $k\geq k_\varepsilon$ we have
\begin{equation*}\begin{split}\|&(\mathfrak m_t(v_k)\otimes 1)(P_F\otimes 1)\xi\|^2 \\
&\leq \sum_h|v^k_h|^2 \exp\left(-\frac{t^2}{2}\|q(h)\|^2+t^2D(q)\right)\\
&= \sum_{h\in F_\varepsilon}|v^k_h|^2 \exp\left(-\frac{t^2}{2}\|q(h)\|^2+t^2D(q)\right)+\sum_{h\in \G\setminus F_\varepsilon}|v^k_h|^2 \exp\left(-\frac{t^2}{2}\|q(h)\|^2+t^2D(q)\right)\\
&\leq  \sum_{h\in F_\varepsilon}\frac{\varepsilon}{2|F_{\varepsilon}|\max_{h\in F_\varepsilon}\exp\left(-\frac{t^2}{2}\|q(h)\|^2+t^2D(q)\right)} \exp\left(-\frac{t^2}{2}\|q(h)\|^2+t^2D(q)\right)+\\ &\quad+\sum_{h\in \G\setminus F_\varepsilon}
|v^k_h|^2\frac{\varepsilon}{2\max_k\|v_k\|^2_\infty}\\
&\leq  \frac{\varepsilon}{2}+\frac{\varepsilon}{2}=\varepsilon, \end{split}\end{equation*} which gives the desired conclusion.
\end{proof}

\section{The proof of Theorem \ref{solidity}}

We begin by proving a general result (Theorem \ref{solidity1} below) describing the ``position'' of all commuting subalgebras in crossed products $L^\infty( X)\rtimes \G$ arising from actions $\G\ca X$ of exact groups $\G$ admitting proper arrays. The strategy of proof will
essentially follow Theorem 4.3 in \cite{PetL2}, formally replacing the family of almost derivations with the one-parameter group $\a_t$ constructed in \S\ref{path}. We note that unlike the proofs of solidity by Popa
\cite{PoFree} and Vaes \cite{Vae}---which also make use of one-parameter automorphism groups---we cannot directly appeal to spectral gap
arguments and must, like Peterson, make fundamental use of Haagerup's criterion for amenability. We state Haagerup's criterion here for the convenience of the reader.

\begin{prop}[Haagerup, Lemma 2.2 in \cite{Haa}]\label{108} Let $M$ be a ${\rm II}_1$ factor. A von Neumann subalgebra $N \subset M$ is amenable
if and only if for every non-zero projection $p\in \Cal Z (N)$ and every finite set of unitaries $F\subset \Cal U(Np)$ we have
\begin{equation}\|\sum_{u\in F}u \otimes\bar u\|_\infty=|F|.\end{equation}
\end{prop}

Using this criterion, in cooperation with the technical results from the previous section, we show the following.

\begin{theorem}\label{solidity1} Let $\G$ be an exact group together with a finite family of subgroups $\mathcal F$. We assume that $\G$ admits an array into a weakly-$\ell^2$ representation that is proper with respect to $\mathcal F$.  Also, let  $\G \ca X$ be a free, ergodic, measure-preserving action of $\G$ on a probability space and denote by $M = L^\infty(X)\rtimes \G$. Then for any diffuse von Neumann subalgebra $A\subset M$, either:
\begin{enumerate} \item $A'\cap M$ is amenable; or, \item $A \preceq_M L^\infty(X)\rtimes \Sigma$, for some $\Sigma\in\mathcal F$. \end{enumerate}
\end{theorem}

\begin{proof} Denote by  $N=A'\cap M$. Assuming that $A \not\preceq_M L^\infty(X)$ we will show that $N$ must be amenable. Let $p \in  \Cal Z(N)$ be a non-zero projection and $F \subset \Cal U(Np)$ be a finite set of unitaries.

For convenience we recall that in Section 2.3 we considered a path of unitaries $V_t \in \fr B(L^2(Y ^\pi)\otimes L^2(X)\otimes \ell^2(\G))$ defined as $V_t(\xi\otimes \eta\otimes \de_\g)=\up_t(\g)\xi\otimes\eta\otimes\de_\g$, for every $\xi\in L^2(Y^\pi),\eta\in L^2(X)$, and $\g\in\G$. We claim that, since $A\not\preceq_M L^\infty(X)\rtimes \Sigma$ for all $\Sigma\in\mathcal F$,  $V_t$ cannot converge uniformly to the identity  on $(Ap)_1$. Indeed, if $A\not\preceq_M L^\infty(X)\rtimes \Sigma$ for all $\Sigma\in\mathcal F$ then, by Popa's intertwining techniques, there exists a sequence of unitaries $v_k\in \Cal U(Ap)$ such that for all $s,t\in \G$ we have that  \begin{equation}\label{110} \|E_{L^\infty(X)\rtimes \Sigma}(u_tv_k u_s)\|_2\ra 0\text{ as }k\ra \infty.\end{equation} Then letting $v_k=\sum_\g v^k_\g u_\g \in  L^{\infty}(X)\rtimes \G $ it is easy to see that
\begin{equation}\label{111}
\begin{split}\nor{e\cdot V_t(v_k)}^2 &=\nor{e\cdot  V_t(\sum_\g v^k_\g\de_\g)}^2 \\
 &= \sum_\g  \|v^k_\g\|_2^2\exp\left(-2t^2\nor{q(\g)}^2\right )\\
 &= \sum_\g  \|E_{L^\infty(X)}(v_k u^*_\g)\|_2^2\exp\left(-2t^2\nor{q(\g)}^2\right ).\end{split}
 \end{equation}
Since $q$ is proper relative to the family $\mathcal F$, an easy computation shows that (\ref{110}) together with (\ref{111}) imply that $\|e\cdot  V_t(v_k)\|$ converges to $0$, as $k\ra \infty $. Hence our claim follows because we have $\nor{V_t(x) - x}\geq
\nor{e\cdot V_t(x) - x}$, for all $x \in Ap$.

In conclusion, one can find a constant $c > 0$ such that for every $t > 0$ there exists $x_t \in (Ap)_1$ such that $\nor{V_t( x_t) - x_t} \geq c$. Let us denote $V_t( x_t)$ by $\zeta_t$ and define $\xi_t$ to be $\zeta_t - e(\zeta_t)$. By Lemma \ref{transversality}, we have that $\displaystyle\nor{\xi_t} \geq \frac{c}{2}$.

Let $E \subset L^\infty(X)\rtimes \G$ be the operator system spanned by $\{p\} \cup F \cup F^\ast$. Since $\G$ is exact and $L^{\infty}(X)$ is abelian (hence, nuclear as a $C^\ast$-algebra), for every $\G$-invariant, separable $C^\ast$-algebra $A$ the reduced crossed product $A\rtimes_{\s,r} \G$ is exact, which implies that $L^{\infty }(X)\rtimes_{\s,r} \G$ is locally reflexive. Therefore, one can find a net $(\vp_i)_{i \in I}$ of contractive completely positive maps $\vp_i: E\to L^\infty(X)\rtimes_{\s,r} \G$ such that $\vp_i \to \id_E$ pointwise-ultraweakly. In fact, by passing to convex combinations of the $\vp_i$'s, we may assume that $\vp_i(u) \to u$ in the strong* topology for all $u\in F$. Fixing $i \in I$, we have that for all $u \in F$

\begin{equation}\label{spectralgap}\begin{split}
&\lim_{t\to 0}\nor{\vp_i(u)\xi_t\vp_i(u^*) - \xi_t}_2 \\
&= \lim_{t\to 0}\nor{(1-e)(\vp_i(u)\zeta_t\vp_i(u^*) - \zeta_t)}_2 \\
&\leq \lim_{t\to 0}\nor{\vp_i(u)J\vp_i(u)J(\zeta_t) - \zeta_t}_2 \\
&\leq  \lim_{t\to 0}\nor{\vp_i(u)V_t\left (\alpha_{-t}(J\vp_i(u)J)-J\vp_i(u)J\right )x_t}_2+ \lim_{t\to 0}\nor{V_t\left (\alpha_{-t}(\vp_i(u))-\vp_i(u)\right )J\vp_i(u)J x_t}_2\\ &\quad + \lim_{t\to 0}\nor{V_t\left (\vp_i(u)J\vp_i(u)J x_t- x_t\right)}_2\\
&\leq  \lim_{t\to 0}\nor{\left (\alpha_{-t}(J\vp_i(u)J)-J\vp_i(u)J\right ) x_t}_2+ \lim_{t\to 0}\nor{\left (\alpha_{-t}(\vp_i(u))-\vp_i(u)\right )J\vp_i(u)J x_t}_2\\ &\quad + \lim_{t\to 0}\nor{\vp_i(u))J\vp_i(u)Jx_t- x_t}_2\\
&\leq 2\lim_{t\to 0}\nor{(\a_t(\vp_i(u)) - \vp_i(u))\cdot e}_\infty+\lim_{t\to 0}\nor{\vp_i(u)x_t\vp_i(u^*) - x_t}_2 \\
&= \lim_{t\to 0}\nor{\vp_i(u)x_t\vp_i(u^*) - x_t}_2 \\
&\leq \lim_{t\to 0}\nor{u x_t u^* -x_t}_2 + 2\lim_{t\to 0}\nor{x_t}_\infty \nor{\vp_i(u) - u}_2 \\
&\leq 2\nor{\vp_i(u) - u}_2
\end{split}
\end{equation}

Given $\varepsilon >0$, let us choose $i \in I$ such that $\displaystyle\sum_{u\in F} \nor{\vp_i(u) - u}_2 \leq \frac{ c\varepsilon}{4}$. Viewing $\mathcal H = L_0^2(Y ^\pi)\otimes L^2(X)\otimes \ell^2(\G)$ in the natural way as a Hilbert $M$-bimodule, we have that $\mathcal H$ is weakly contained in the coarse bimodule over $M$ (cf.\ Lemma 5.1 in \cite{OPCartanII}). Hence, the induced representation $\pi': M\otimes_{\rm alg} M^o\to \fr B(\mathcal H)$ given by $\pi'(x\otimes y^o)\xi = x\xi y$ extends in the minimal tensor norm. Thus, by the calculations above there exists $t >0$ such that
\begin{equation*}\label{solidityineq}\begin{split} \nor{\sum_{u\in F}u \otimes \bar u}_\infty &\geq \nor{\sum_{u\in F}\vp_i(u) \otimes \overline{\vp_i(u)}}_\infty \\
&\geq \frac{\nor{\sum_{u\in F}\vp_i(u)\xi_t\vp_i(u^*)}_2}{\nor{\xi_t}_2} \geq \abs{F} - \varepsilon.
\end{split}\end{equation*}
Hence, by Haagerup's criterion we have that $N$ is amenable, a contradiction.
\end{proof}

\begin{proof}[Proof of Theorem \ref{solidity}] Let $\G $ be an exact group belonging to the class $\mathcal {QH}_{\rm reg}$ and  let $A\subset L\G$ be a diffuse subalgebra. Applying the previous theorem in the case when $X$ consists is a point, we have either $A'\cap M$ is amenable or $A\preceq_M \mathbb C$. However, the second case is impossible since $A$ is diffuse; hence, it follows that $L\G$ is solid.
\end{proof}

\section{The proof of Theorem \ref{strongsolidity} and corollaries}

The above techniques for proving solidity can be upgraded to more general situations.  Specifically,  we obtain a result  describing all weakly compact embeddings in the crossed product von Neumann algebras arising from actions of hyperbolic groups (Theorem \ref{controlweakembedding}). Though our approach follows the general outline of the proof of Theorem B in \cite{OPCartanII} there are substantial technical issues which arise in working with deformations from quasi-cocycles. For instance, one has to confront a lack of traciality in the spectral gap arguments. However, we believe that, by dealing with these obstacles, the techniques  developed bring new insight in proving structural results for these factors.

To state the main theorem we need to recall the notion of weak compactness introduced by Ozawa and Popa in \cite{OPCartanI}. Briefly, a given inclusion of finite von Neumann algebras $B\subseteq Q$ is said to be a weakly compact embedding if the natural action by conjugation of the normalizer $\mathcal N_Q(B)\ca B$ is \emph{weakly compact}. This means that there exists  a net of positive unit vectors $(\eta_n)_{n\in\mathbf N}$ in $L^2(Q) \otimes L^2(\bar Q)$ which simultaneously satisfies the following relations:
\begin{enumerate}

\item $\nor{\eta_n - (v \otimes \bar v)\eta_n}\to 0$, for all $v \in \Cal U(B)$;

\item $\nor{[u \otimes \bar u,\eta_n]}\to 0$, for all $u \in \Cal N_Q(B)$; and

\item $\ip{(x \otimes 1)\eta_n}{\eta_n} = \tau (x) = \ip{(1 \otimes \bar x)\eta_n}{\eta_n}$, for all $x \in Q$.

\end{enumerate}

\begin{theorem}\label{controlweakembedding} Let $\G$ be an exact group which admits a proper quasi-cocycle into a weakly-$\ell^2$ representation. Let $\G\curvearrowright X$ be a measure-preserving action on a probability space and denote by $M=
L^\infty(X)\rtimes\G$. If $P\subset M$ is a weakly compact embedding, then one of the following possibilities must hold:
\begin{enumerate}\item $P\preceq_M L^\infty(X)$; \item $\mathcal N_M(P)''$ is amenable.\end{enumerate}
\end{theorem}

 \begin{proof} We will denote by  $N =\Cal  N_M(P)''$ and fix $p\in \mathcal Z(N'\cap M)$ a projection. The general strategy of the proof to show that the assumption $P\not\preceq_M L^\infty(X)$ implies that $Np$ is amenable. By assumption $P\subset M$ is weakly compact, so there exists a net of unit vectors $(\eta_n)_{n\in\mathbf N}$ in $L^2(M) \otimes L^2(\bar M)$ as above.

Let $\mathcal H = L^2_0(Y^\pi)\otimes L^2(X)\otimes \ell^2(\G)$ which as we remarked in the proof of Theorem \ref{solidity1} is weakly contained as an $M$-bimodule in the coarse bimodule. Fixing $t>0$ we consider the unitary $V_t$ associated with the quasi-cocycle $q$ as defined in the previous section. Next denote by $\tilde \eta_{n,t}=(V_t\otimes 1)(p\otimes 1)\eta_n$, $\zeta_{n,t} = (e\otimes1)\tilde\eta_{n,t}= (e\cdot  V_t \otimes 1)(p\otimes 1)\eta_n$, and $\xi_{n,t} = \tilde\eta_{n,t} - \zeta_{n,t}=(e^\perp\otimes 1)\tilde\eta_{n,t}\in \mathcal H\otimes L^2(M)$. Using these notations we first prove the following result which is a technical adaptation of Proposition \ref{improuvconv2}.

\begin{lem}\label{improuvconv} Let $\fr m_t$ be the Schur multiplier associated to the proper quasi-cocycle $q$ on $\G$ and let $\tilde{\fr m}_t = \id_X\otimes \fr m_t$. Let $v_k\in M$ be a bounded sequence of elements such that for all $x,y\in M$ we have  $\|E_{L^\infty(X)}(xv_ky)\|_2 \ra  0 $ as $k\ra \infty$. Then for every $t>0$ and every finite set $F\subset \G$ we have that
\begin{equation*}\lim_{k\ra \infty}\left (\sup_{n}\|(\tilde {\mathfrak m}_t(v_k)\otimes 1)(P_F\otimes 1)\zeta_{n,t}\|\right )=0.\end{equation*}
Here we denoted by $P_F$ the orthogonal projection from $L^2(M)$ onto the linear span of the set $\{L^2(X)\otimes \de_\a \,|\, \a\in F\}$.\end{lem}

\begin{proof} Let $(p\otimes 1)\eta_n=\sum_s a^n_s\otimes b_s$ where $\{b_s\}_{s\in \mathbb N} $ is an orthonormal basis of $L^2(M)$ and $a^n_s=\sum_r a^{s,n}_r\otimes \de_r$. Without loss of generality we may assume that the set $F$ consists of one element, i.e., $F=\{\g\}$. Let $v_k =\sum_h v^k_hu_h$ be the Fourier expansion, where $v^k_h\in L^\infty (X)$. Then applying the formula for $\mathfrak m_t$ together with $(P_F\otimes 1)(\zeta_{n,t})=(P_F\otimes 1)(e\cdot V_t\otimes 1)((p\otimes 1)\eta_n)=\sum_s \left((\exp(-t^2\|q(\g)\|^2)a^{s,n}_\g\right )\otimes \de_\g)\otimes b_s$ and $\exp(-2t^2\|q(\g)\|^2)\leq 1$, we obtain the following formulas

\begin{equation}\label{55}\begin{split}&\|(\tilde{\mathfrak m}_t(v_k)\otimes 1)(P_F\otimes 1)\zeta_{n,t}\|^2\\ &=\sum_s\| eV_t v_kV_{-t}((\exp(-t^2\|q(\g)\|^2)a^{s,n}_\g)\otimes \de_\g)\|^2\\&=\exp(-2t^2\|q(\g)\|^2)\sum_s\|\sum_h eV_tv^k_hu_hV_{-t}(a^{s,n}_\g\otimes \de_\g)\|^2\\
&\leq \sum_s\|\sum_h(\tau(\upsilon_t(h\g)\sigma_h (\upsilon_{-t}(\g)))1)\otimes (v^k_h \s_h(a^{s,n}_\g))\otimes \de_{h\g}\|^2\\
&=\sum_s\sum_h\|v^k_h\s_h(a^{s,n}_\g)\|^2 |\tau(\upsilon_t(h\g)\sigma_h (\upsilon_{-t}(\g)))|^2\\
&=\sum_s\sum_h\|v^k_h\s_h(a^{s,n}_\g)\|^2 \exp\left(-t^2\|q(h\g)-\pi_h (q(\g))\|^2\right)
\end{split}\end{equation}

 \noindent Furthermore, by the quasi-cocycle relation, the last term in the equation above is smaller than

  \begin{equation}\label{1000}\leq \sum_{s}\sum_h\|v^k_h\s_h(a^{s,n}_\g)\|^2 \exp\left(-\frac{t^2}{2}\|q(h)\|^2+t^2D(q)\right).\end{equation}

Applying the identity $\|(x\otimes 1)(p\otimes 1)\eta_n\|=\|xp\|_2$ for all elements of the form  $x=v^k_hu_h$, $h\in \G$, a basic calculation shows that   $\sum_s\sum_r \|v^k_h\s_h(a^{s,n}_r)\|^2=\|v^k_hu_hp\|_2^2$ for all $k\in \mathbb N$, $h\in \G$, and $n$. In particular, this implies that, for all $k\in \mathbb N$, $h\in \G$, and $n$, we have

\begin{eqnarray*}\label{1000000}\sum_s \|v^k_h\s_h(a^{s,n}_\g)\|^2\leq\|v^k_h\|_2^2.\end{eqnarray*}

\noindent  Using these estimates we see that the expression (\ref{1000}) is smaller than

 \begin{equation}\label{1000000}\leq \sum_h\|v^k_h\|_2^2\exp\left(-\frac{t^2}{2}\|q(h)\|^2+t^2D(q)\right).\end{equation}

Fix $\varepsilon>0$. Since the map $q$ is proper, there exists a finite subset $F_\varepsilon\subset \G$ such that
 $\displaystyle\frac{2}{t^2}\ln\left(\frac{2\exp \left(t^2 D(q)\right)}{\varepsilon}\right) \leq \|q(h)\|^2$ for all $h\in \G\setminus F_\varepsilon$. This obviously implies
 that \begin{equation}\label{56}\exp\left(-\frac{t^2}{2}\|q(h)\|^2+t^2D(q)\right)\leq \frac{\varepsilon}{2\max_k\|v_k\|^2_\infty},\end{equation} for all $h\in \G\setminus F_\varepsilon$. Since for all $x,y\in M$,  the sequence $\|E_{L^\infty(X)}(xv_k y)\|_2$ converges  to $0$ as $k$ approaches $\infty$ and $F_\varepsilon$ is finite (depending only on $\varepsilon$ and $q$), then making a suitable choice for $x$ and $y$ one can find $k_\varepsilon\in\mathbb N$
 such that, for all $k \geq k_\varepsilon$ and all $h\in F_\varepsilon$, we have
 \begin{equation}\label{57}\|v^k_h\|_2\leq \left(\frac{\varepsilon}{2|F_{\varepsilon}|\max_{h\in F_\varepsilon}\exp\left(-\frac{t^2}{2}\|q(h)\|^2+t^2D(q)\right)}\right)^{\frac{1}{2}}.\end{equation}

 Altogether, relations  (\ref{55}), (\ref{1000000}), (\ref{56}), and (\ref{57}) show that, for all $k\geq k_\varepsilon$, we have

 \begin{equation*}\begin{split}&\sup_{n}\|(\tilde{\mathfrak m}_t(v_k)\otimes 1)(P_F\otimes 1)\zeta_{n,t}\|^2\\
 &\leq \sum_{h\in F_{\varepsilon}}\|v^k_h\|_2^2 \exp\left(-\frac{t^2}{2}\|q(h)\|^2+t^2D(q)\right)+\sum_{h\in \G\setminus F_{\varepsilon}}\|v^k_h\|_2^2 \exp\left(-\frac{t^2}{2}\|q(h)\|^2+t^2D(q)\right)\\
 &\leq  \sum_{h\in F_{\varepsilon}}\frac{\varepsilon}{2|F_{\varepsilon}|\max_{h\in F_\varepsilon}\exp\left(-\frac{t^2}{2}\|q(h)\|^2+t^2D(q)\right)} \exp\left(-\frac{t^2}{2}\|q(h)\|^2+t^2D(q)\right)  + \\  & +    \sum_{h\in \G\setminus F_\varepsilon}\|v^k_h\|_2^2\left( \frac{\varepsilon}{2\max_k\|v_k\|_\infty}\right)\\
&\leq  \frac{\varepsilon}{2}+\frac{\varepsilon}{2}=\varepsilon, \end{split}\end{equation*} which gives the desired conclusion.
\end{proof}

\noindent Using the notations introduced at the beginning of the proof we show next the following inequality:
\begin{lem}\label{OP4.11} \begin{equation*}\Lim_n\nor{\xi_{n,t}} \geq \frac{5}{12}\|p\|_2,
\end{equation*}

\noindent where ``Lim'' is a limit along a non-principal ultrafilter.
\end{lem}

\begin{proof}Using the triangle inequality multiple times, we have that

\begin{equation*}\begin{split}\nor{\tilde\eta_{n,t} - (e\circ\a_{t}(v) \otimes \bar v)\zeta_{n,t}}
&\leq \nor{\tilde\eta_{n,t} - (e\cdot\a_{t}(v) \otimes \bar v)\tilde \eta_{n,t}} +\nor{\xi_{n,t}}\\
&\leq \nor{\zeta_{n,t} - (e\cdot \a_{t}(v) \otimes \bar v)\tilde \eta_{n,t}} +2\nor{\xi_{n,t}}\\
&\leq \nor{\tilde\eta_{n,t} - (\a_{t}(v) \otimes \bar v)\tilde \eta_{n,t}} +2\nor{\xi_{n,t}}\\
&\leq \nor{\eta_{n} - (v \otimes \bar v)\eta_{n}} +2\nor{\xi_{n,t}}, \end{split}
\end{equation*}
for all $v \in\Cal  U(P)$ and all $n\in \mathbf N$.

\noindent Consequently, since by (3) we have  $\nor{\tilde\eta_{n,t}}=\|p\|_2$, using the triangle inequality again we get

\begin{equation}\label{OP4.8}\nor{(e\circ\a_{t}(v) \otimes \bar v)\zeta_{n,t}}\geq \|p\|_2-2\nor{\xi_{n,t}}-\nor{\eta_{n} - (v \otimes \bar v)\eta_{n}}. \end{equation}

Next we consider the operator $e\cdot V_t\otimes 1$ from $L^2(M)\otimes L^2(M)$ to $L^2(M)\otimes L^2( M)$. Since $q$ is proper one can check that there exists a finite subset $F\subset \G$ such that $\displaystyle \nor{(P_F\otimes 1)(e\cdot
V_t\otimes 1)-e\cdot  V_t\otimes 1}_{\infty}\leq \frac{1}{6}\|p\|_2$: here $P_F$ denotes the orthogonal projection on the linear $\text{span of }L^\infty(X)F$.
Hence using the formula $e\circ \a_t\circ e = e\circ \tilde{\mathfrak m}_t$ together with relation (\ref{OP4.8}) and the triangle inequality, we obtain

\begin{equation}\label{OP4.10} \nor{(\tilde{\mathfrak m}_t(v)\otimes 1)( P_F\otimes 1)\zeta_{n,t}} \geq \frac{5}{6}\|p\|_2-2\nor{\xi_{n,t}}-\nor{\eta_{n} - (v \otimes \bar v)\eta_{n}},
\end{equation}
for all $v \in  \Cal U(P)$ and all $n\in\mathbf N$.

This further implies that

\begin{equation}\label{OP4.100} \sup_{s}\nor{(\tilde{\mathfrak m}_t(v)\otimes 1)( P_F\otimes 1)\zeta_{s,t}} \geq \frac{5}{6}\|p\|_2-2\nor{\xi_{n,t}}-\nor{\eta_{n} - (v \otimes \bar v)\eta_{n}},
\end{equation}
for all $v \in \Cal U(P)$ and all $n\in\mathbf N$.

Taking  ``$\Lim$'', an arbitrary limit along a non-principal ultrafilter, above and applying relation (1) we obtain
\begin{equation}\label{OP4.7} \sup_{s}\nor{(\tilde{\mathfrak m}_t(v)\otimes 1)( P_F\otimes 1)\zeta_{s,t}}\geq \frac{5}{6}\|p\|_2-2\Lim_n\nor{\xi_{n,t}},\end{equation}
for all $v\in \mathcal U(P)$. This shows that the limit $\Lim_n\|\xi_{n,t}\|\geq \frac{5}{12}\|p\|_2$. Indeed, since $P\not\preceq L^{\infty }(X)$, by Popa's intertwining techniques there exists a sequence of unitaries $v_s\in\mathcal U(P)$ such that for all $x,y\in M$ we have $\|E_{L^\infty(X)}(xv_ky)\|_2\ra 0$, as $k\ra \infty$. Applying inequality (\ref{OP4.7}) for each $v_k$ and taking the limit as $k\ra \infty$ then  Lemma \ref{improuvconv} implies that the left side of (\ref{OP4.7}) is $0$ and we get the desired conclusion.\end{proof}

Following the same argument as in Theorem B of \cite{OPCartanII}, we define a state $\psi_t$ on $\mathcal N =\mathfrak B(\mathcal
H)\cap\rho(M^{op})'$. Explicitly, $\psi_t(x)=\Lim_n\frac{1}{\|\xi_{n,t}\|^2}\langle (x\otimes 1)\xi_{n,t},\xi_{n,t}\rangle$ for every $x \in
\mathcal N $. Next we prove the following technical result

\begin{lem}\label{almostcomm1}For every $\varepsilon>0$ and every finite self-adjoint set $K\subset L^{\infty}(X)\rtimes_{\s,r} \G$ with $dist_{\|\cdot\|_2} (y,(N)_1)\leq \varepsilon$ for all $y\in K$ one can find $t_\varepsilon >0$ and a finite set $L_{K,\varepsilon}\subset  \mathcal N_M(P)$ such that
\begin{equation}\label{101}|\langle((yx-xy)\otimes 1)\xi_{n,t},\xi_{n,t}\rangle | \leq  10\varepsilon+2\sum_{v\in L_{K,\varepsilon}}\|[v\otimes \bar v,\eta_n]\| ,\end{equation} for all  $y\in K$, $\|x\|_\infty\leq 1$, $t_\varepsilon>t>0$, and $n$.\end{lem}

\begin{proof} Fix $\varepsilon>0$ and $y\in K$. Since $N=\mathcal N_M(P)''$ by the Kaplansky density theorem  there exists a finite set $F_y=\{v_i\}\subset \mathcal N_M(P)$ and scalars $\mu_i$
such that $\|\sum_i\mu_i v_i\|_{\infty}\leq 1$ and \begin{equation}\label{102}\|y-\sum_i\mu_i v_i\|_2\leq \varepsilon.\end{equation}

\noindent Also using Proposition \ref{deform-ineq} one can find a positive number $t_\varepsilon>0$ such that,  for all $t_\varepsilon>t>0$, we have
\begin{equation}\label{103}\begin{split}&\nor{(y - \a_{-t}(y))\cdot e}_\infty \leq \varepsilon;\\&\nor{(JyJ - \a_{-t}(JyJ))\cdot e}_\infty \leq \varepsilon.\end{split}
\end{equation}

 Next we will proceed in several steps to show inequality (\ref{101}). First we fix $t_\varepsilon>t>0$. Then, using the triangle inequality in combination with $\|x\|_{\infty}\leq 1$, (\ref{103}), and the $M$-bimodularity of $1-e=e^{\perp}$, we see that
\begin{equation*}\label{OP4.12}\begin{split}
& |\langle(x\otimes 1)\xi_{n,t},(y^*\otimes 1)\xi_{n,t}\rangle-\langle (x y\otimes
1)\xi_{n,t},\xi_{n,t}\rangle|\\
&\leq \| (\a_{-t}(y^*)-y^*)\cdot e\|_{\infty}+ |\langle(x\otimes
1)\xi_{n,t_\varepsilon},(e^{\perp}V_ty^*p\otimes 1)\eta_{n}\rangle-\langle (xy\otimes 1)\xi_{n,t},\xi_{n,t}\rangle|\\
&\leq \varepsilon+ |\langle(x\otimes 1)\xi_{n,t},(e^{\perp}V_{t}y^*p\otimes 1)\eta_{n}\rangle-\langle
(xy\otimes1)\xi_{n,t},\xi_{n,t}\rangle|\end{split}\end{equation*}

 Furthermore, the Cauchy-Schwarz inequality together with (3) and (\ref{102}) enable us to see that the last quantity above is smaller than

\begin{equation*}\begin{split}
&\leq \varepsilon+\|((y^*p-\sum_i \bar\mu_i v^*_i)p\otimes 1)\eta_n\|+|\sum_i \mu_i\langle(x\otimes
1)\xi_{n,t},(e^{\perp}V_{t} v^*_ip\otimes 1)\eta_{n}\rangle-\langle (xy\otimes 1)\xi_{n,t},\xi_{n,t}\rangle|\\
&\leq2 \varepsilon+|\sum_i  \mu_i\langle(x\otimes\bar v^*_i)\xi_{n,t},(e^{\perp}V_{t} pv^*_i\otimes \bar v^*_i)\eta_{n}\rangle-\langle (xy\otimes1)\xi_{n,t},\xi_{n,t}\rangle|\end{split}\end{equation*}

To this end we notice that, since $\eta_n$ is a positive vector and $J$ is an isometry then for all $z\in M$ we have \begin{equation}\label{1004}\|\eta_n(z\otimes 1)\|=\|J(z^*\otimes 1)J\eta_n\|=\|(z^*\otimes 1)\eta_n\|=\|z^*\|_2=\|z\|_2\end{equation} Using this identity in combination with (\ref{102}) and  $v_i$ being a unitary, we see that the last quantity above is smaller than

\begin{equation}\label{1003}\begin{split} &\leq 2\varepsilon+\sum_i\|[v^*_i\otimes \bar v^*_i,\eta_n]\|+|\sum_i
 \mu_i\langle(x \otimes \bar v^*_i)\xi_{n,t},(e^{\perp}V_{t} p\otimes 1)(\eta_nv^*_i\otimes \bar v^*_i)\rangle-\langle (xy\otimes 1)\xi_{n,t},\xi_{n,t}\rangle|\\
 &\leq 2\varepsilon+\sum_i\|[v^*_i\otimes \bar v^*_i,\eta_n]\|+|\langle(x \otimes 1)\xi_{n,t},(e^{\perp}V_{t} p\otimes 1)(\eta_n(\sum_i
 \bar \mu_iv^*_i)\otimes 1)\rangle-\langle (xy\otimes 1)\xi_{n,t},\xi_{n,t}\rangle|\\
 &\leq 3\varepsilon+\sum_i\|[v^*_i\otimes \bar v^*_i,\eta_n]\|+|\langle(x \otimes 1)\xi_{n,t},(e^{\perp}V_{t} p\otimes 1)(\eta_n(y^*\otimes 1))\rangle-\langle (xy\otimes 1)\xi_{n,t},\xi_{n,t}\rangle|\end{split}\end{equation}

Next we observe that  using the second part of (\ref{103}) and $V_t$ being a unitary we have \begin{eqnarray*}\begin{split}&\|V_{t} \otimes 1((p\otimes 1)\eta_n(y^*\otimes 1))-(V_{t}\otimes 1(( p\otimes 1)\eta_n))(y^*\otimes 1)\|\\
&=\|(V_{t}JyJ \otimes 1)(p\otimes 1)\eta_n-(JyJV_{t}\otimes 1)( p\otimes 1)\eta_n\|\\
&\leq\|(V_{t}JyJ \otimes 1)(p\otimes 1)\eta_n-(JyJV_{t}\otimes 1)( p\otimes 1)\eta_n\|\\
&\leq\|((JyJ -\a_{-t}(JyJ))\otimes 1)(p\otimes 1)\eta_n\|\\
&\leq\|(JyJ -\a_{-t}(JyJ))\cdot e\|_\infty\leq \varepsilon
\end{split}
\end{eqnarray*}

Therefore applying this estimate two times we see that the last expression in (\ref{1003}) is smaller that

 \begin{equation*}\begin{split} \label{1005}&\leq 4\varepsilon+\sum_i\|[v^*_i\otimes \bar v^*_i,\eta_n]\|+|\langle(x \otimes 1)\xi_{n,t},((e^{\perp}V_{t} p\otimes 1)\eta_n)(y^*\otimes 1)\rangle-\langle (xy\otimes 1)\xi_{n,t},\xi_{n,t}\rangle|\\
&= 4\varepsilon+\sum_i\|[v^*_i\otimes \bar v^*_i,\eta_n]\|+|\langle((x \otimes 1)\xi_{n,t})(y\otimes 1),(e^{\perp}V_{t} p\otimes 1)
\eta_n\rangle-\langle (xy\otimes 1)\xi_{n,t},\xi_{n,t}\rangle|\\
&\leq 5\varepsilon+\sum_i\|[v_i\otimes \bar v_i,\eta_n]\|+|\langle(xe^{\perp}V_{t}p\otimes 1)(\eta_{n}(y\otimes 1)),(e^{\perp}V_{t} p\otimes 1)\eta_n\rangle-\langle (xy\otimes 1)\xi_{n,t},\xi_{n,t}\rangle|\end{split}\end{equation*}

Using  (\ref{102}),  (\ref{1004}), $v_i$ being a unitary in combination with triangle inequality  we see that the last quantity above is smaller than

\begin{equation*}\begin{split}&\leq 6\varepsilon+\sum_i\|[v_i\otimes \bar v_i,\eta_n]\|+\\
&|\sum_i\mu_i\langle(xe^{\perp}V_{t}p\otimes 1)(\eta_{n}(v_i\otimes \bar v_i)),(e^{\perp}V_{t} p\otimes \bar v_i)\eta_n)\rangle-\langle (xy\otimes 1)\xi_{n,t},\xi_{n,t}\rangle|\\
&\leq 6\varepsilon+2\sum_i\|[v_i\otimes \bar v_i,\eta_n]\|+\\
&|\sum_i\mu_i\langle(xe^{\perp}V_{t}p\otimes 1)(v_i\otimes \bar v_i)\eta_{n},(e^{\perp}V_{t} p\otimes \bar v_i)\eta_n\rangle-\langle (xy\otimes 1)\xi_{n,t},\xi_{n,t}\rangle|\\
&= 6\varepsilon+2\sum_i\|[v_i\otimes \bar v_i,\eta_n]\|+\\
&\langle(xe^{\perp}V_{t}p(\sum_i\mu_iv_i)\otimes 1)\eta_{n},(e^{\perp}V_{t} p\otimes 1)\eta_n\rangle-\langle (xy\otimes 1)\xi_{n,t},\xi_{n,t}\rangle|\end{split}\end{equation*}

Then (\ref{102}) together with (3), the Cauchy-Schwarz inequality, and $\|x\|_\infty\leq 1$ show that the last quantity above is smaller than

\begin{equation*}\begin{split} &\leq 7\varepsilon+2\sum_i\|[v_i\otimes \bar v_i,\eta_n]\|+\langle(xe^{\perp}V_{t}py\otimes 1)\eta_{n},(e^{\perp}V_{t} p\otimes 1)\eta_n\rangle-\langle (xy\otimes 1)\xi_{n,t},\xi_{n,t}\rangle\\
&\leq 7\varepsilon+2\sum_i\|[v_i\otimes \bar v_i,\eta_n]\|+\|((V_{t}py- yV_tp)\otimes 1)\eta_n\|\end{split}\end{equation*}

Finally, using  (\ref{103}) together with (3) and the initial assumption $dist_{\|\cdot\|_2} (y,(N)_1)\leq \varepsilon$ we obtain that the last expression above is smaller than

\begin{equation*}\begin{split}&\leq 7\e+2\sum_i\|[v_i\otimes \bar v_i,\eta_n]\|+\|((V_{t}(py- yp))\otimes 1)\eta_n\|+\|((V_{t}y-yV_t)\otimes 1)(p\otimes 1)\eta_n)\|\\
&\leq 7\varepsilon+2\sum_i\|[v_i\otimes \bar v_i,\eta_n]\|+\|((py- yp)\otimes 1)\eta_n\|+\|(y-\a_{-t}(y))\cdot e\|_{\infty} \\
& \leq 8\varepsilon+2\sum_i\|[v_i\otimes \bar v_i,\eta_n]\|+\|py- yp\|_2\\
& \leq 10\varepsilon+2\sum_i\|[v_i\otimes \bar v_i,\eta_n]\|.\end{split}\end{equation*}

In conclusion, (\ref{101}) follows from the previous inequalities  by taking $L_{K,\varepsilon}=\cup_{y\in K} F_y$.\end{proof}

\begin{lem}\label{almostcomm2}For every $\varepsilon>0$ and any finite set $F_0\subset \mathcal U(N)$ there exist a finite set $F_0\subset F\subset M$, a c.c.p. map  $\varphi_{F,\varepsilon}: span (F)\ra L^{\infty}(X)\rtimes_{\s,r} \G $, and $t_\varepsilon>0$ such that
\begin{equation}\label{OP4.16}|\psi_{t_\varepsilon}(\varphi_{F,\varepsilon}(up)^*x\varphi_{F,\varepsilon}(up))-\psi_{t_{\varepsilon}}(x)|\leq  116\varepsilon,\end{equation} for all $u\in F_0$ and $\|x\|_\infty\leq 1$.
\end{lem}

\begin{proof} Fix $\varepsilon>0$. Denote by $F=\{up,u^*p\}\cup F_0\cup F^*_0 $ and $E= span (F)$. By local reflexivity, we may choose a c.c.p.\ map $\vp_{F,\varepsilon}: E\to  L^{\infty}(X)\rtimes_{\s,r} \G$ such
that for all $u\in F$
\begin{equation}\label{104}\nor{\vp_{F,\varepsilon}(up)-up}_2\leq \varepsilon.\end{equation}
This shows in particular that $dist_{\|\cdot\|_2}(\vp_{F,\varepsilon}(up), (N)_1)\leq \varepsilon$ for all $u\in F$. Therefore, applying the previous lemma for
$K=\{\vp_{F,\varepsilon}(up) : u\in F  \}\subset  L^{\infty}(X)\rtimes_{\s,r} \G$, there exists a $t_\varepsilon>0$ and a finite set $K'\subset \mathcal N_M(P)$ such that,  for all
$u\in F$, all $\|x\|_{\infty}\leq 1$, and all $n$, we have

\begin{equation}\label{106}|\langle((\varphi_{F,\varepsilon}(up)^*x\varphi_{F,\varepsilon}(up)-x\varphi_{F,\varepsilon}(up)\varphi_{F,\varepsilon}(up)^*)\otimes 1)\xi_{n,t_\varepsilon},\xi_{n,t_\varepsilon}\rangle |\leq 10\varepsilon+ 2\sum_{v\in K'}\|[v\otimes \bar v,\eta_n]\|\end{equation}

\noindent Also using Proposition \ref{deform-ineq}, after shrinking $t_\varepsilon$ if necessary, we can assume in addition that for all $u\in F$ we
have

\begin{equation}\label{105}\nor{(\vp_{F,\varepsilon}(up) - \a_{-t_\varepsilon}(\vp_{F,\varepsilon}(up)))\cdot e}_\infty \leq \varepsilon.
\end{equation}

 Hence, using triangle inequality together with (\ref{106}) and the Cauchy-Schwarz inequality, we have that

 \begin{equation*}\begin{split} & |\langle\left(\varphi_{F,\varepsilon}(up)^*x\varphi_{F,\varepsilon}(up)\otimes 1\right)\xi_{n,t_\varepsilon},\xi_{n,t_\varepsilon}\rangle-\langle (x\otimes 1)\xi_{n,t_\varepsilon},\xi_{n,t_\varepsilon}\rangle| \\&\leq
10\varepsilon+2\sum_i\|[v\otimes \bar v,\eta_n]\|+|\langle(x(\varphi_{F,\varepsilon}(up)\varphi_{F,\varepsilon}(up)^*-1)\otimes 1)\xi_{n,t_\varepsilon},\xi_{n,t_\varepsilon}\rangle\\
&\leq 10\varepsilon+2\sum_v\|[v\otimes \bar v,\eta_n]\|+ \|x\|_\infty \|(\varphi_{F,\varepsilon}(up)\varphi_{F,\varepsilon}(up)^*-1)\otimes 1)\xi_{n,t_\varepsilon}\|\\
&\leq 10\varepsilon+2\sum_v\|[v\otimes \bar
v,\eta_n]\|+2\|(\a_{-t_\varepsilon}(\vp_{F,\varepsilon}(up))-\vp_{F,\varepsilon}(up))\cdot e\|_{\infty}\\&\quad+\|(V_{t_\varepsilon}(\varphi_{F,\varepsilon}(up)\vp_{F,\varepsilon}(up)^*-1)p\otimes
1)\eta_{n}\|\end{split}\end{equation*}

\noindent Furthermore, using (\ref{104}) together with the Cauchy-Schwarz inequality, (3) and (\ref{105}) we see that the last quantity above is
smaller than

\begin{equation*}\begin{split}
&\leq 12\varepsilon+2\sum_v\|[v\otimes \bar v,\eta_n]\| +\|((\varphi_{F,\varepsilon}(up)\vp_{F,\varepsilon}(up)^*-1)p\otimes 1)\eta_{n}\|\\
&\leq 12\varepsilon+2\sum_v\|[v\otimes \bar v,\eta_n]\| +\|\varphi_{F,\varepsilon}(up)\vp_{F,\varepsilon}(up)^*-p\|_2\\
&\leq 12\varepsilon+2\sum_v\|[v\otimes \bar v,\eta_n]\| +2\|\varphi_{F,\varepsilon}(up)-up\|_2\\
&\leq 20\varepsilon+2\sum_v\|[v\otimes \bar v,\eta_n]\|.
\end{split}
\end{equation*}

Altogether, the above sequence of inequalities shows that
\begin{equation*}
|\langle(\varphi_{F,\varepsilon}(up)^*x\varphi_{F,\varepsilon}(up)\otimes 1)\xi_{n,t_\varepsilon},\xi_{n,t_\varepsilon}\rangle-\langle (x\otimes
1)\xi_{n,t_\varepsilon},\xi_{n,t_\varepsilon}\rangle|\leq 20\varepsilon+2\sum_{v\in K'}\|[v\otimes \bar v,\eta_n]\|,
\end{equation*}
and combining this with (2) and Lemma \ref{OP4.11} we obtain
\begin{equation*}\begin{split}|\psi_{t_\varepsilon}(\varphi_{F,\varepsilon}(up)^*x\varphi_{F,\varepsilon}(up))-\psi_{t_{\varepsilon}}(x)|&\leq \Lim_n\left ( \frac{20\varepsilon+2\sum_v\|[v\otimes \bar v,\eta_n]\|}{\|\xi_{n,t_\varepsilon}\|^2}\right )\\ & \leq \frac{20\varepsilon}{\left (\frac{5}{12}\right)^{2}}<116\varepsilon,\end{split}\end{equation*}
which finishes the proof.\end{proof}

For the remaining part of the proof we mention that one can use Haagerup criterion to show that $Np$ is amenable. In fact the reasoning in Theorem B in \cite{OPCartanII} applies verbatim in our case and we  leave the details to the reader.\end{proof}

\begin{proof}[Proof of Theorem \ref{strongsolidity}] Let $\G$ be an i.c.c.\ group which is weakly amenable and admits a proper quasi-cocycle into the left-regular representation, and consider $A\subset L\G=M$ a diffuse amenable subalgebra. By Theorem B in \cite{OzCBAP} it follows that $A$ is weakly compact in $L\G$. Also, weak amenability implies that $\G$ is exact, cf.\ Theorem 12.4.4 in \cite{BrOz}. Hence, applying the previous theorem for the case when $X$ consists of a point, we obtain that either $\mathcal N_M(A)''$ is amenable or $A\preceq_M \mathbb C$. Since $A$ is diffuse the second case is impossible and therefore $L\G$ is strongly solid.
\end{proof}

\begin{proof}[Proof of Corollary B.\ref{Sp}] In the case that $\G$ is hyperbolic, a result of Ozawa shows that $\G$ is weakly amenable \cite{OzHyp}.
In the case that $\G$ is a lattice in ${\rm Sp}(n,1)$, choose a co-compact lattice $\La < {\rm Sp}(n,1)$. We have that $\La$ is Gromov
hyperbolic; hence, by \cite{MMS} $\La$ admits a proper quasi-cocycle in to the left-regular representation. A result of Shalom (Theorem 3.7 in \cite{Shalom}) shows that
$\G < {\rm Sp}(n,1)$ is integrable, thus $\ell^1$-measure equivalent to $\La$. As explained in the proof of Proposition \ref{prop:qh_permanence}, item (4), this implies that $\G$ also admits a proper quasi-cocycle into the left-regular representation. The work of
Cowling and Haagerup \cite{CoHa} shows that ${\rm Sp}(n,1)$ is weakly amenable, which implies, by an unpublished result of Haagerup, that any
lattice in ${\rm Sp}(n,1)$ is also weakly
 amenable (cf.\ \cite{OzCBAP}). Therefore, the hypotheses of Theorem \ref{strongsolidity} are satisfied.

If $\G$ is an exact, weakly amenable group which admits a proper quasi-cocycle into a weakly-$\ell^2$ representation, then for any profinite, free, ergodic measure-preserving action $\G \ca X$ on a standard probability space, $M = L^\infty(X)\rtimes \G$ is a weakly amenable ${\rm II}_1$ factor. If $A\subset M$ is a Cartan subalgebra, then the normalizing algebra $\mathcal N_M(A)''$ is obviously non-amenable and therefore, by Theorem \ref{controlweakembedding}, we must have that $A\preceq_M L^\infty(X)$. Hence, by Appendix A of \cite{PoBe}, there exists $u\in \Cal U(M)$ such that $uAu^* = L^\infty(X)$. Next, if $\La\ca Y$ is a free, ergodic measure-preserving action of a countable discrete group $\La$ on a standard probability space $Y$ such that $\theta: L^\infty(Y)\rtimes \La \ra L^\infty(X)\rtimes \G$ is an isomorphism of von Neumann algebras, then we may assume that $\theta(L^\infty(Y)) =L^\infty(X)$. In particular, the actions $\G\ca X$ and $\La\ca Y$ are orbit equivalent and by Theorem A of \cite{IoaCSR} it follows that $\G\ca
X$ and $\La\ca Y$ are virtually conjugate.
\end{proof}

\begin{proof}[Proof of Corollary B.\ref{generalizadams}] Let $\G\ca X$ and $\Lambda \ca Y$ be two orbit equivalent actions. Therefore one can find an isomorphism $\theta: M =L^\infty(Y)\rtimes \La\ra L^\infty (X)\rtimes \G$ such that $\theta(L^\infty(Y))=L^\infty (X)$. Let $\Sigma< \Lambda$ be an infinite amenable subgroup and we assume by contradiction that its normalizing group $N_\La(\Sigma)$ is non-amenable. From the assumption it follows that $\La$ is weakly amenable and therefore the action by conjugation of $N_\La(\Sigma)$ on $L\Sigma$ is weakly compact and so is the action by conjugation of $\theta (N_\La(\Sigma))$ on $\theta(L\Sigma)$ \cite{OzCBAP}. Since $N_\La(\Sigma)$ is non-amenable the previous theorem implies that $\theta (L\Sigma)\preceq_M L^\infty (X)$ and since $\theta(L^\infty(Y))=L^\infty (X)$ this is equivalent to $L\Sigma\preceq L^\infty (Y)$. This however is impossible. Indeed, by intertwining techniques this implies that one can find finitely many elements $x_i, y_i\in M $  and $C>0$ such that

\begin{equation}\label {112}\sum_i\|E_{L^\infty(Y)}(x_ivy_i)\|^2_2\geq C, \text{  for all }v \in  \mathcal U(L\Sigma). \end{equation}

By shrinking $c$ a little we can assume that for all $i$ the elements $x_i$ and $y_i$ have finite supports in $\G$. Therefore the union $F\subset \G$ of all these supports is still a finite set and so is $F^{-1}F^{-1}$. Since $\Sigma$ is infinite one can find  $\g\in \Sigma\setminus F^{-1}F^{-1}$. A simple computation shows that all elements $x_iu_\g y_i$ are supported on elements different than the identity and hence  $E_{L^\infty(Y)}(x_ivy_i)=0$ which contradicts  (\ref{112}).  \end{proof}
\begin{remark} Note that the same proof above shows that all i.c.c.\ groups in the orbit equivalence class of an i.c.c.\ hyperbolic group give rise to strongly solid factors. This should be compared with an observation of the second author and Peterson (Remark 6.4 in \cite{PeSi}) on orbit-equivalence and strong solidity of group factors for weakly amenable groups with the ``$L^2$-Haagerup property''.\end{remark}

\section{The proof of Theorem \ref{prime} and corollary} This last section in the main body of the paper contains the proof of Theorem \ref{prime} on the uniqueness of decompositions of group von Neumann algebras of products of groups in $\Cal {QH}_{\rm reg}$ into prime factors. Our proof is designed to circumvent a technical subtlety in the proof of Theorem 6.1 in \cite{PetL2} on the norm estimates of fusion products of certain vectors $\tilde\delta^i_\alpha(x^i_\alpha)$---specifically,  whether these vectors are uniformly bounded from below, so that the estimate ``$\nor{\xi_\alpha}_2 \geq c^m$'' is achieved.  Essentially, we will be using Theorem \ref{solidity1} together with the fact that there is a well-defined way to take a tensor product of arrays as explained in Proposition \ref{productarrays}.

\begin{proof}[Proof of Theorem \ref{prime}] Notice that via a canonical isomorphism we can view $M =  L\G_1\oo \dotsb\oo L\G_n=L \G$ where $\G= \G_1\times\G_2\times \cdots\times \G_n$. By assumption, for each $1\leq i\leq n$, there exists an array $q_i :  \G_i\to  \mathcal H_i$ into some weakly-$\ell^2$ unitary representation $\pi_i : \G_i\ra \mathcal U ( \mathcal H_i)$. Consider the tensor product representation $\pi : \G\ra \mathcal U(\mathcal H)$,  where $\mathcal H = \mathcal H_1\otimes   \mathcal H_2\otimes\cdots\otimes  \mathcal H_n$ and $\pi=\pi_1\otimes \pi_2\otimes \cdots \otimes \pi_n$, and notice that since $\mathcal H_i$ is weakly-$\ell^2$ then $\mathcal H$ is weakly-$\ell^2$ for $\G$. Setting $\hat\G_i$ to be the kernel of the canonical projection from $\G$ onto $\G_i$, using Proposition \ref{productarrays} inductively, one can construct an array $ q: \G\to \mathcal H$ which is proper with respect to the family $\{\hat \G_i : 1\leq i\leq n\}$.

Now, suppose that $B \subset L\G$ is a ${\rm II}_1$ subfactor whose relative commutant $N = B' \cap L\G$ is a non-amenable factor. Therefore by Theorem \ref{solidity1} there exists $1\leq j\leq n$ such that $B \preceq_M L\hat \G_j$. The result then follows by appealing to Proposition 12 of \cite{OPPrime}.\end{proof}

For the remaining corollary we fix the following notation. Given a subset  $F\subset \{1,\dotsc, n\}$, we denote by $\hat \G_F$ the subgroup of $\G=\G_1\times \dotsb \times \G_n$ which consists of all elements with trivial $i$-th coordinate, for all $i\in F$.

\begin{proof}[Proof of Corollary C \ref{meprime}] Suppose that $m\leq n$. Since $\G_1\times \dotsb \times  \G_n\sim_{ME} \La_1\times \dotsb \times \La_m$, there exists an isomorphism $\psi: A\rtimes (\G_1\times \cdots \times \G_n)\ra (B\rtimes (\La_1\times \cdots \times \La_m))^t$ such that $\psi(A)=B^t$. For simplicity, we will assume that $t=1$, and we denote $\G=\G_1\times \cdots \times \G_n$, $\La =\La_1\times\cdots\times \La_m$, $M= A\rtimes \G$, and $N=  B\rtimes \La$. Also, throughout the proof, for every subset $F\subset \{1,\dotsc, n \}$ and $K\subset \{1,\ldots,m\}$, we define $\hat M_F=A\rtimes \hat\G_F$ and $\hat N_K=B\rtimes \hat \La_K$.

To begin, we prove that for every proper subset $F\subset \{1,\dotsc, n \}$ there exists $K\subset \{1,\dotsc, m \}$ such that $\abs{F}= \abs{K}$ and \begin{equation}\label{500}\phi(L(\hat \G_F))\preceq_N \hat N_K.\end{equation}

First, we notice that the same argument as in the proof of previous theorem shows that if $P\subset N$ is a diffuse subfactor such that there exists a non-amenable subfactor $Q\subset P'\cap N$, then one can find $1\leq l\leq m$ such that $P\preceq_N \hat N_{l}$. In particular, this shows our claim when $F$ consists of a single element. To get the general case we will proceed by induction on the cardinality of $F$. To explain the inductive step, fix a proper subset $F$ of $\{1,\ldots, n\}$ together with an element  $k\in F$.  By assumption, there exists $K'\subset \{1,\ldots, m\}$ with $\abs{K'} = \abs{F}-1$ such that  \begin{equation}\label{506}\phi(L(\hat \G_{F\setminus \{k\}}))\preceq_N \hat N_{K'}.\end{equation}

Therefore, since all $\G_i$ are i.c.c., one can find projections $p_1\in \phi(L(\hat \G_{F}))$, $p_2\in \phi(L(\G_{k}))$, $q\in \hat N_{K'}$, and an injective homomorphism \[\theta: (p_1\otimes p_2)\,\phi(L(\hat \G_{F\setminus \{k\}}))\,(p_1\otimes p_2)\ra  q\hat N_{K'}q.\] Next, we notice that $\theta(p_1\otimes p_2\,\phi(L(\hat \G_{F}))\,p_1\otimes p_2)$ and $\theta(p_1\otimes p_2\,\phi(L(\G_k)\,p_1\otimes p_2)$ are non-amenable, commuting subfactors of $\theta(p_1\otimes p_2)\, \hat N_{K'}\,\theta(p_1\otimes p_2)$; thus, applying the same argument as before, there exists an element $j\in \{1,\ldots, m\}\setminus K'$ such that  \begin{equation}\label{507}\theta(p_1\otimes p_2\,\phi(L(\hat \G_{F}))\,p_1\otimes p_2)\preceq_{\hat N_{K'}} \hat N_{K'\cup \{j \}}.\end{equation} Finally, by Remark 3.8 in \cite{vaes}, relations (\ref{506}) and (\ref{507}) imply that  $\phi(L(\hat \G_{F}))\preceq_N \hat N_{K'\cup \{j \}}$, which concludes the inductive step and the proof of (\ref{500}).

Notice that (\ref{500}) automatically implies that $m=n$. Indeed, if $m\geq n+1$, then applying the statement for any subset $F\subset \{1,\dotsc, n \}$ with $\abs{F} =m-1$, we get that $\phi(L(\hat \G_F))\preceq_N B\rtimes \La_l$ for some $1\leq l\leq m$ which obviously contradicts Theorem \ref{solidity1}.
Also, (\ref{500}) implies that for every $1\leq i\leq n$ there exists $1\leq \pi(i)\leq n$ such that $\phi(L( \G_i))\preceq_N B\rtimes \La_{\pi(i)}$.
Notice that since $\phi(A)=B$, for all $1\leq i\leq n$, we also have that  \begin{equation}\label{501}\phi( A \rtimes  \G_i)\preceq_N B\rtimes \La_{\pi(i)}.\end{equation}

Applying the same procedure for $\phi^{-1}$, for every $1\leq i\leq n$, one can find $1\leq\rho(\pi(i))\leq n$ such that  $\phi^{-1}( B \rtimes  \La_{\pi(i)})\preceq_M A\rtimes \G_{\rho( \pi(i))}$; equivalently, \begin{equation}\label{505}B \rtimes  \La_{\pi(i)}\preceq_N \phi(A\rtimes \G_{\rho( \pi(i))}),\end{equation} for all $1\leq i\leq n$. Combining this with (\ref{501}) and using that $\phi(A\rtimes \G_i)$ is an irreducible subfactor of $N$, we obtain $ \phi(A\rtimes \G_{i})\preceq_N\phi(A\rtimes \G_{\rho( \pi(i))})$. In particular, this implies that $\rho(\pi(i))=i$, for all $1\leq i\leq n$; hence, $\pi$ is permutation of $\{1,\ldots,n\}$. Therefore, using (\ref{501}) and (\ref{505}) together with Proposition 8.4 in \cite{ipp}, one can find unitaries $u_i\in \mathcal U(N)$ such that \begin{equation}\label{504}u_i\phi( A \rtimes  \G_i)u^*_i= B\rtimes \La_{\pi(i)}.\end{equation} This further gives that $\phi_{u_i}={\rm Ad}(u_i)\circ\phi$ is an isomorphism from $A\rtimes \G_i$ onto $B\rtimes \La_{\pi(i)}$  which satisfies
\begin{equation*}\phi_{u_i}(a)u_i=u_i\phi(a),\end{equation*}
for all $a\in A$.

Next, for $N = B\rtimes (\La_{\pi(i)}\times \hat \La_{\pi(i)})$, we consider the Fourier decomposition $u=\sum_{\lambda\in \hat \La_{\pi(i)}} y_\lambda v_\lambda$ with $y_\lambda \in B\rtimes \La_{\pi(i)}$ and, using the above equation, there exists a nonzero element $y_\lambda\in B\rtimes \La_{\pi(i)}$ such that for all $a\in A$ we have
\begin{equation}\label{503}\phi_{u_i}(a)y_\lambda=y_\lambda \delta _\lambda (\phi(a)).\end{equation}
Note that since $B=\phi(A)$ is a maximal abelian subalgebra of $N$, (\ref{503}) implies that $y^*_\lambda y_\lambda \in B$. Furthermore, taking the polar decomposition $y_\lambda =w_\lambda\abs{y_\lambda}$ with $w_\lambda$ a partial isometry, we conclude that

\begin{equation*}\phi_{u_i}(a)w_\lambda=w_\lambda \delta_\lambda (\phi(a)),\end{equation*}
for all $a\in A$.
This shows in particular $\phi_{u_i} (A)\prec_{B\rtimes \La_{\pi(i)}} \phi(A)$, and since $B=\phi(A)$ and $\phi_{u_i}(A)$ are Cartan subalgebras of $B\rtimes \La_{\pi(i)}$, then by Theorem A2 in \cite{PoBe} there exists a unitary $u'_i\in B\rtimes \La_{\pi(i)}$ such that \[u'_i\psi_u(A){u'_i}^*=\phi(A)=B.\] Finally, letting $x_i=u'_i u_i\in \mathcal N_N(B)$, by (\ref{504}) the map ${\rm Ad}(x_i)\circ \phi$ implements an isomorphism between $A\rtimes \G_i$ and $B\rtimes \La_{\pi(i)}$, identifying the Cartan subalgebras $A$ and $B$; thus, $\G_i \sim_{ME} \La_{\pi(i)}$ for all $1\leq i\leq n$. \end{proof}

\appendix

\section{Amenable Actions, Exactness, and Local Reflexivity}\label{sec:biexact}

\begin{definition}[Anantharaman-Delaroche and Renault \cite{ADRen}, cf.\ \cite{HiRo}]\label{defn:amenable} Let $\G$ be a countable discrete group and $\G \ca X$ be an action of $\G$ by homeomorphisms on a compact topological space $X$. The action $\G \ca X$ is said to be \textit{amenable} if there exists a sequence $(\xi_n)$ of continuous maps $\xi_n: X\to \ell^2(\G)$ such that $\xi_n\geq 0$, $\nor{\xi_n(x)}_2 = 1$, for all $x\in X$, $n\in \bb N$, and
\begin{equation}\label{eq:amenable} \sup_{x\in X}\nor{\la_\g(\xi_n(x)) - \xi_n(\g x)}_2\to 0,
\end{equation}
for all $\g\in \G$.
\end{definition}

\begin{prop}[Higson and Roe \cite{HiRo}]\label{defn:propA} A countable discrete group $\G$ has Guoliang Yu's \textit{property A} \cite{YuA} if and only if $\G$ acts amenably on its Stone--\v{C}ech boundary
$\beta'\G = \beta\G \setminus \G$.
\end{prop}

\noindent Property A is equivalent, cf. \cite{Roe}, to the nuclearity of ${\rm C}_u^\ast(\G)$ which is, in turn, equivalent to the exactness of
${\rm C}_\la^\ast(\G)$ by a result of Ozawa \cite{OzExact}.

\begin{definition}\label{defn:locreflex} A $\rm C^\ast$-algebra $A$ is said to be \emph{locally reflexive} if for every finite-dimensional operator system $E \subset A^{**}$, there exists a net $(\vp_i)_{i\in I}$ of contractive completely positive (c.c.p.) maps $\vp_i: E\to A$ which converges to the identity in the pointwise-ultraweak topology.
\end{definition}

For the purposes of this paper, the crucial property implied by exactness is that ${\rm C}_\la^\ast(\G)$ is a locally reflexive ${\rm
C}^\ast$-algebra, cf.\ \cite{BrOz}, Chapter 9.\\

\section{A proof of Proposition \ref{prop:SL}}\label{sec:SL}

The aim of this appendix is to provide an elementary proof that $\G = \bb Z^2\rtimes {\rm SL}(2,\bb Z)$ belongs to the class $\Cal{QH}_{\rm reg}$. Appealing to
Theorem \ref{solidity} then furnishes an alternate proof of the solidity of $L\G$, the main result of \cite{OzExample}. As in \cite{OzExample},
our proof will make use of the amenability of the natural action of ${\rm SL}(2,\bb Z)$ on ${\rm SL}(2,\bb R)/T \cong \bb RP^1$, where $T$ is
the group of upper-triangular $2\times 2$ real matrices.

To begin, note that $\G_0 = {\rm SL}(2,\bb Z)$ admits a proper cocycle $b: \G_0\to \ell^2(\G_0)$ with respect to the left-regular
representation. By Proposition \ref{prop:symm-array}, we may replace $b$ with a proper, symmetric array $b'$ into the left-regular representation. Let $\pi$ be the representation of $\G$ on $\ell^2(\G_0)$ obtained by pulling the left-regular representation of $\G_0$ back
along the quotient $\G\twoheadrightarrow \G/\bb Z^2 \cong \G_0$, so that $\pi$ is weakly contained in the left-regular representation of $\G$.
Let $p: \bb Z^2\setminus\{(0,0)\}\to \bb RP^1$ be the projection defined by $p((x,y)) = x/y$, and note that $p$ is equivariant with respect to
the natural actions of ${\rm SL}(2,\bb Z)$ on $\bb Z^2$ and $\bb RP^1$.

Given a sequence of continuous maps $\xi_n: \bb RP^1\to \ell^2(\G_0)$ satisfying Definition \ref{defn:amenable}, define the maps $\xi_n': \bb
Z^2\to \ell^2(\G_0)$ by \[\xi_n'(z) = \xi_n(p(z)),\] for $z = (z_1,z_2)\in \bb Z^2\setminus \{(0,0)\}$, and $\xi_n'(z) = 0$,
otherwise. Note that for any $a\in \bb Z^2$ we have
\begin{equation}\label{eq:limsup} \limsup_{z\to \infty}\nor{\xi_n'(z) - \xi_n'(z+a)}_2 = 0,
\end{equation}
for all $n\in \bb N$.

Now, consider finite, symmetric generating subsets $S'\subset \G_0$ and $S''\subset \bb Z^2$. Define $S_1 = S' \cup S''$ and $S_{k+1} = S_k
\cup (S_1)^{k+1}$ for all $k\in \bb N$. By equations \ref{eq:amenable} and \ref{eq:limsup}, there exists an increasing sequence of finite,
symmetric subsets $F_1\subset F_2\subset \dotsb \subset F_k\subset \dotsb \subset \bb Z^2$ such that $\bigcup^{\infty}_{k=1} F_k =\mathbb Z^2$
and a subsequence $(n_k)$ such that
\begin{equation} \sup_{s\in S_k}\sup_{g\in \bb Z^2\setminus F_k}\nor{\pi_s(\xi_{n_k}'(g)) - \xi_{n_k}'(s\cdot g)}_2\leq \frac{1}{2^k},
\end{equation}
where $s\cdot g$ is the natural $\G$-action on $\bb Z^2$. Define a map $\del: \bb Z^2\to \ell^2(\bb N; \ell^2(\G_0)) = \Cal H$ by $\del(z)(k) =
\xi_{n_k}'(z)$, if $z\not\in F_k$, and $0$, otherwise. It is then straightforward to check that $\del$ is proper, symmetric, and boundedly
$\G$-equivariant. For $(z,\g)\in \bb Z^2\rtimes {\rm SL}(2,\bb Z)$ we define the map $q((z,\g)) = b'(\g)\oplus \del(z)\in \ell^2(\G_0)\oplus
\Cal H$. It is easy to see that $q$ is a proper, symmetric array into the weakly-$\ell^2$ representation $\pi\oplus \pi^{\oplus\infty}$. Thus, $\bb Z^2\rtimes
{\rm SL}(2,\bb Z)\in \Cal{QH}_{\rm reg}$ and we are done.

\begin{question} Does $\bb Z^2\rtimes {\rm SL}(2,\bb Z)$ admit a proper, \emph{anti-symmetric} array into a weakly-$\ell^2$ representation?
\end{question}

\bibliographystyle{amsplain}

\begin{thebibliography}{10}

\bibitem{AdHyp} Scot Adams: \textit{Indecomposability of equivalence relations generated by word hyperbolic groups}, Topology {\bf 33} (1994), no. 4, 785--798.

\bibitem{ADRen} Claire Anantharaman-Delaroche and Jean Renault: \textit{Amenable groupoids}, Monographies de L'Enseignement Math\'ematique {\bf 36} (2000), 196 pp., L'Enseignement Math\'ematique, Geneva.

\bibitem{BV} Bachir Bekka, Pierre de la Harpe, and Alain Valette: \textit{Kazhdan's property (T)}, New Mathematical Monographs {\bf 11} (2008), xiv+472 pp., CUP, Cambridge.

\bibitem{BrOz} Nathanial P. Brown and Narutaka Ozawa: \textit{${\rm C}^\ast$-algebras and finite-dimensional approximations}, Graduate Studies in Mathematics, vol. 88, AMS, Providence, RI.

\bibitem{BuMo} Marc Burger and Nicolas Monod: \textit{Continuous bounded cohomology and applications to rigidity theory}, Geom. Funct. Anal. {\bf 12} (2002), 219--280.

\bibitem{CP} Ionut Chifan and Jesse Peterson: \textit{Some unique group-measure space decomposition results}, preprint (2010).

\bibitem{CSU} Ionut Chifan, Thomas Sinclair, and Bogdan Udrea: \textit{On the structure of $\rm II_1$ factors of negatively curved groups, II. Actions by product groups}, preprint (2011).

\bibitem{CoHa} Michael Cowling and Uffe Haagerup: \textit{Completely bounded multipliers of the Fourier algebra of a simple Lie group of real rank one}, Invent. Math. {\bf 96} (1989), no. 3, 507--549.

\bibitem{CoZi} Michael Cowling and Robert J.\ Zimmer: \textit{Actions of lattices in ${\rm Sp}(1,n)$}, Ergodic Theory Dynam. Systems {\bf 9} (1989), no. 2, 221--237.

\bibitem{FurOE} Alex Furman: \textit{Orbit equivalence rigidity}, Ann. Math. (2) {\bf 150} (1999), 1083--1108.

\bibitem{Haa} Uffe Haagerup: \textit{Injectivity and decomposition of completely bounded maps}, Operator algebras and their connections with topology and ergodic theory (Bu\c{s}teni, 1983), Lecture Notes in Mathematics {\bf 1132} (1985), 170--222, Springer, Berlin.

\bibitem{HiRo} Nigel Higson and John Roe: \textit{Amenable group actions and the Novikov conjecture}, J. Reine Angew. Math. {\bf 519} (2000), 143--153.

\bibitem{IoaCSR} Adrian Ioana: \textit{Cocycle superrigidity for profinite actions of property (T) groups}, Duke Math. J., {\bf 157} (2011), no. 2, 337--367.

\bibitem{ipp} Adrian Ioana, Jesse Peterson, and Sorin Popa: \textit{Amalgamated free products of weakly rigid factors and calculation of their symmetry groups}, Acta Math. {\bf 200} (2008), no. 1, 85--153.


\bibitem{Min} Igor Mineyev: \textit{Straightening and bounded cohomology of hyperbolic groups}, Geom. Funct. Anal. {\bf 11} (2001), no. 4, 807--839.

\bibitem{MMS} Igor Mineyev, Nicolas Monod, and Yehuda Shalom: \textit{Ideal bicombings for hyperbolic groups and applications}, Topology {\bf 43} (2004), no. 6, 1319--1344.

\bibitem{Monod} Nicolas Monod: \textit{Continuous bounded cohomology of locally compact groups}, Lecture Notes in Mathematics {\bf 1758} (2001), Springer, Berlin.

\bibitem{MScocycle} Nicolas Monod and Yehuda Shalom: \textit{Cocycle superrigidity and bounded cohomology for negatively curved spaces}, J. Differential Geom. {\bf 67} (2004), no. 3, 395--455.

\bibitem{MSoe} Nicolas Monod and Yehuda Shalom: \textit{Orbit equivalence rigidity and bounded cohomology}, Ann. Math. (2) {\bf 164} (2006), no. 3, 825--878.

\bibitem{Osin} Denis Osin: \textit{Relatively hyperbolic groups: Intrinsic geometry, algebraic properties, and algorithmic problems}, Memoirs Amer. Math. Soc. {\bf 179} (2006), no. 843, vi+100 pp.

\bibitem{OzExact} Narutaka Ozawa: \textit{Amenable actions and exactness for discrete groups}, C. R. Acad. Sci. Paris S\'er. I Math. {\bf 330} (2000), no. 8, 691--695.

\bibitem{OzSolid} Narutaka Ozawa: \textit{Solid von Neumann algebras}, Acta Math., {\bf 192} (2004), 111--117.

\bibitem{OzKurosh} Narutaka Ozawa: \textit{A Kurosh type theorem for type ${\rm II}_1$ factors}, Int. Math. Res. Not. (2006) Art. ID 97560.

\bibitem{OzHyp} Narutaka Ozawa: \textit{Weak amenability of hyperbolic groups}, Groups Geom. Dyn. {\bf 2} (2008), no. 2, 271--280.

\bibitem{OzExample} Narutaka Ozawa: \textit{An example of a solid von Neumann algebra}, Hokkaido Math. J. {\bf 38} (2009), 567--571.

\bibitem{OzCBAP} Narutaka Ozawa: \textit{Examples of groups which are not weakly amenable}, ArXiv e-prints, December 2010.

\bibitem{OPPrime} Narutaka Ozawa and Sorin Popa: \textit{Some prime factorization results for type $II_1$ factors
}, Invent. Math., {\bf 156} (2004), 223--234.

\bibitem{OPCartanI} Narutaka Ozawa and Sorin Popa: \textit{On a class of {${\rm II}\sb 1$} factors with at most one {C}artan subalgebra {I}},
Ann. of Math. (2), {\bf 172} (2010), no. 1, 713--749.

\bibitem{OPCartanII} Narutaka Ozawa and Sorin Popa: \textit{On a class of {${\rm II}\sb 1$} factors with at most one {C}artan
  subalgebra {II}}, Amer. J. Math., {\bf 132} (2010), no. 3, 841--866.

\bibitem{PetL2} Jesse Peterson: \textit{$L\sp 2$-rigidity in von {N}eumann algebras}, Invent. Math., {\bf 175} (2009), no. 2, 417--433.

\bibitem{PeSi} Jesse Peterson and Thomas Sinclair: \textit{On cocycle superrigidity for Gaussian actions}, Ergodic Theory Dynam. Systems {\bf 32} (2012), 249--272.

\bibitem{PoBe} Sorin Popa: \textit{On a class of type ${\rm II}_1$ factors with Betti numbers invariants}, Ann. Math. (2) {\bf 163} (2006), 809--899.

\bibitem{PoICM} Sorin Popa: \textit{Deformation and rigidity for group actions and von Neumann algebras}, International Congress of Mathematicians. Vol. I (2007), 445--477, Eur. Math. Soc., Z\"urich.

\bibitem{PoFree} Sorin Popa: \textit{On Ozawa's property for free group factors}, Int. Math. Res. Not. (2007), no. 11, 10pp.

\bibitem{PoSG} Sorin Popa: \textit{On the superrigidity of malleable actions with spectral gap}, J. Amer. Math. Soc. {\bf 21} (2008), 981--1000.

\bibitem{PVhypcartan} Sorin Popa and Stefaan Vaes: \textit{Unique Cartan decomposition for $\rm II_1$ factors arising from arbitrary actions of hyperbolic groups}, preprint (2012).


\bibitem{Roe} John Roe: \textit{Lectures on coarse geometry}, University Lecture Series {\bf 31} (2003), vii+175 pp., AMS, Providence, RI.

\bibitem{Sako} Hiroki Sako: \textit{Measure equivalence rigidity and bi-exactness of groups}, J. Funct. Anal. {\bf 257} (2009), no. 10, 3167--3202.

\bibitem{Shalom} Yehuda Shalom: \textit{Rigidity, unitary representations of semisimple groups, and fundamental groups of manifolds with rank one transformation groups}, Ann. Math. (2) {\bf 152} (2000), no. 1, 113--182.

\bibitem{Sin} Thomas Sinclair: \textit{Strong solidity of group factors from lattices in SO(n,1) and
  SU(n,1)}, J. Funct. Anal. {\bf 260} (2011), 3209--3221.


\bibitem{Tho} Andreas Thom: \textit{Low degree bounded cohomology invariants and negatively curved groups}, Groups Geom. Dyn. {\bf 3} (2009), no. 2, 343--358.

\bibitem{vaes}Stefaan Vaes: \textit{Explicit computations of all finite index bimodules for a family of $\rm II_1$ factors}, Ann. Sci. \'Ec. Norm. Sup\'er. (4) {\bf 41} (2008), no. 5, 743--788.

\bibitem{Vae} Stefaan Vaes: \textit{One-cohomology and the uniqueness of the group measure space decomposition of a ${\rm II}_1$ factor}, ArXiv e-prints, December 2010.


\bibitem{Voic} Dan-Virgil Voiculescu: \textit{The analogues of entropy and of Fisher's information measure in free probability theory: the absence of Cartan subalgebras}, Geom. Funct. Anal. {\bf 6} (1996), no. 1, 172--199.

\bibitem{YuA} Guoliang Yu: \textit{The coarse Baum--Connes conjecture for spaces which admit a uniform embedding into Hilbert space}, Invent. Math. {\bf 139} (2000), no. 1, 201--240.


\end{thebibliography}

\end{document}